\documentclass[10pt,oneside,reqno]{amsart}
\usepackage{amscd,amssymb}

\usepackage{AKstyle}

\usepackage{graphicx}
\theoremstyle{plain}
\begin{document}

\title[Affine Linear Sieve]
{
Almost
Prime Pythagorean Triples in Thin Orbits
}
\author{Alex Kontorovich}
\thanks{
Kontorovich is partially supported by  NSF grants DMS-0802998 and DMS-0635607, and the Ellentuck Fund at IAS}
\email{alexk@math.brown.edu}
\address{Mathematics department, Brown University, Providence, RI and
 Institute for Advanced Study, Princeton, NJ}
\author{Hee Oh}
\thanks{
Oh is partially supported by NSF grant DMS-0629322.
}
\email{heeoh@math.brown.edu}
\address{Mathematics department, Brown University, Providence, RI and Korea
 Institute for Advanced Study, Seoul, Korea}

\begin{abstract}
For the ternary quadratic form $Q(\bx)=x^2+y^2-z^2$ and a non-zero Pythagorean triple
${\bx}_0\in \z^3$ lying on the cone $Q({\bx})=0$,
we consider an orbit $\Or={\bx}_0 \G$ of a 
finitely generated subgroup
 $\G<\SO_Q(\z)$ with critical exponent 
exceeding 
$1/2$.

We find infinitely many Pythagorean triples
 in $\Or$ whose hypotenuse, area, and product of side lengths have few prime factors, where ``few'' is explicitly quantified. We also compute the asymptotic of the number of such Pythagorean triples of
  norm at most $T$, up to bounded constants.
\end{abstract} \maketitle
\tableofcontents

\section{Introduction}\label{intro}

\subsection{The Affine Linear Sieve}

In \cite{BourgainGamburdSarnak2006},
 Bourgain, Gamburd, and Sarnak introduced
 the 
 Affine Linear Sieve,
 which
 extends some classical sieve methods to thin orbits of non-abelian group actions.
 Its input is a pair $(\Or,F)$, where

 \begin{enumerate}
 \item
$\Or$ is a discrete orbit,
$\Or=\bx_{0}\cdot\G$,
generated by a discrete subgroup $\G$ of a 
linear group $G$.
It  is called ``thin'' if  the volume of $\G\bk G$ is infinite; and
\item $F$ is a polynomial, taking integer values on $\Or$.
 \end{enumerate}

\

Given the pair $(\Or,F)$, the Affine Linear Sieve attempts to output a number $R=R(\Or,F)$
as small as possible
so that 
there are infinitely many integers $n\in F(\Or)$, with $n$ having at most $R$ prime factors.
\\

A special case of their main result is the following.

\begin{thm}[\cite{BourgainGamburdSarnak2006,BourgainGamburdSarnak2008}]
Let $G< \GL_{n}(\br)$ be
a 
$\Q$-form of
$\SL_{2}$, and let $\G$ be a non-elementary%
\footnote{Recall that a
discrete subgroup $\G<\SL(2,\R)$ is elementary if and only if it has a cyclic subgroup
of finite index.}
 subgroup of $G\cap \GL_n(\Z)$. Let $\Or$ be an orbit ${\bx}_{0} \G$
 for some ${\bx}_{0}\in\Z 
^{n}\setminus\{0\}$ and
$F$ any polynomial which is integral on $\Or$.
Then
there exists a number
$$
R=R(\Or,F)<\infty
$$
such that there are infinitely many $\bx\in\Or$ with  $F(\bx)$ having at most $R$ prime factors.
Moreover the set of such $\bx$ is Zariski dense in the Zariski closure of $\Or$.
\end{thm}

\begin{rmk}
{\rm
As described in \cite[\S2]{BourgainGamburdSarnak2008},
Lagarias gave
 evidence that the result above may be false if one drops
  the condition
  that
    $\G$ is non-elementary.
 }
 \end{rmk}

For various special cases of $(\Or,F)$, one can  say more than just $R<\infty$; one can
give explicit, ``reasonable'' values of $R(\Or,F)$. This 
was achieved with some restrictions in \cite{MyThesis,Kontorovich2009}, and it is our 
present goal to improve the results there in a more general setting. 

In order to remove local obstructions which would increase $R$ for trivial reasons,
we will impose the strong primitivity condition on $(\Or, F)$.
\begin{Def}\label{defSP}
{\rm For a subset
 $\Or\subset \Z ^n $ and a polynomial $F(x_1, \cdots, x_n)$
  taking integral values on $\Or$,
the pair $(\Or,F)$ is called {\it strongly primitive} if for every integer $q\ge2$ there is an ${\bx}\in\Or$ such that
$$
F({\bx})\neq 0 \quad (\mod q).
$$
}
\end{Def}

\begin{figure}
\hskip-2.5in\includegraphics[width=2.5in]{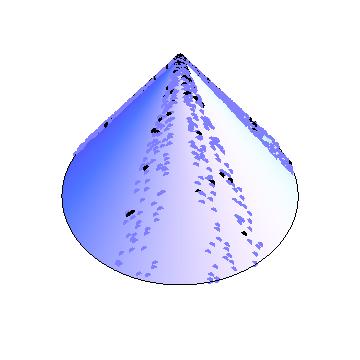}
\vskip-2.25in
\beann
\hskip2in&
\includegraphics[width=.1in]{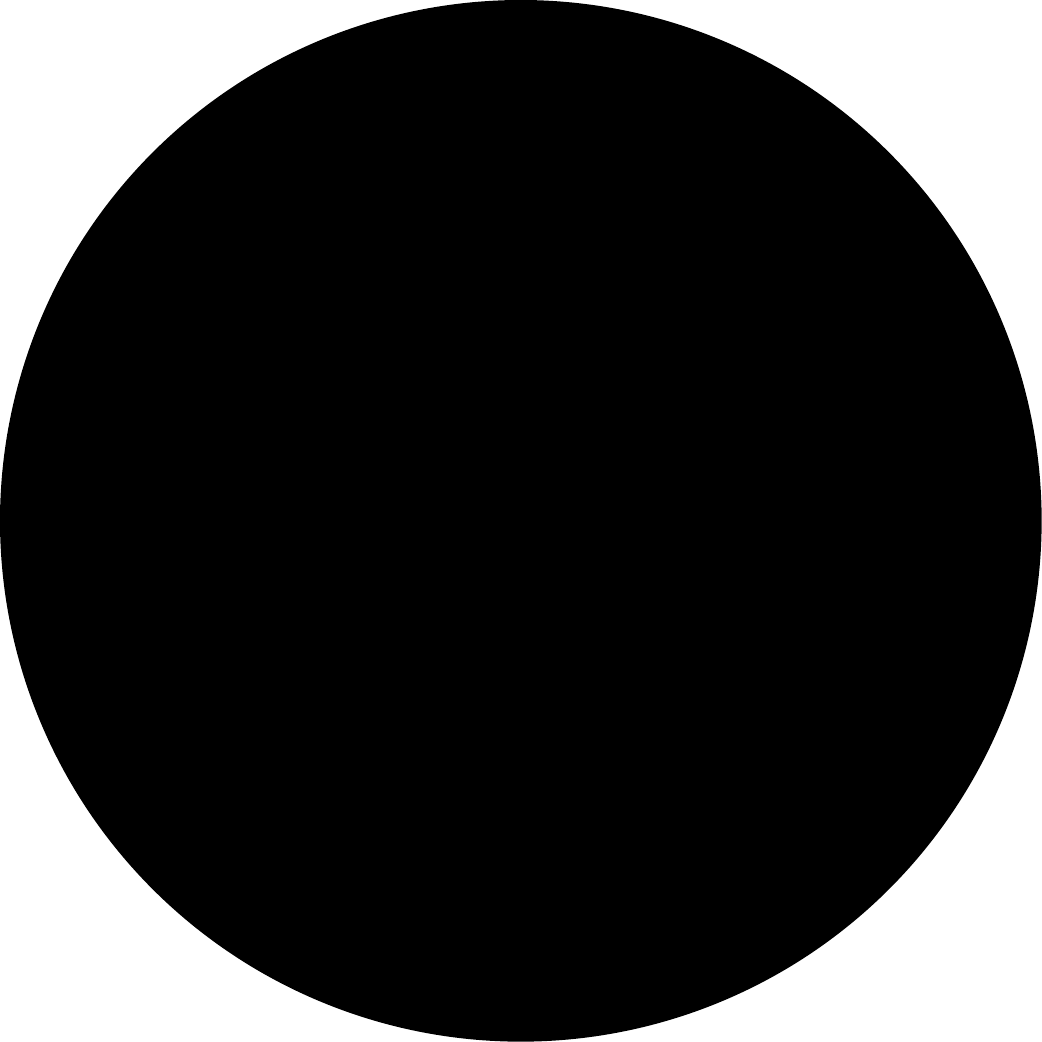}&
\text{Hypotenuse is prime} \\&
\includegraphics[width=.1in]{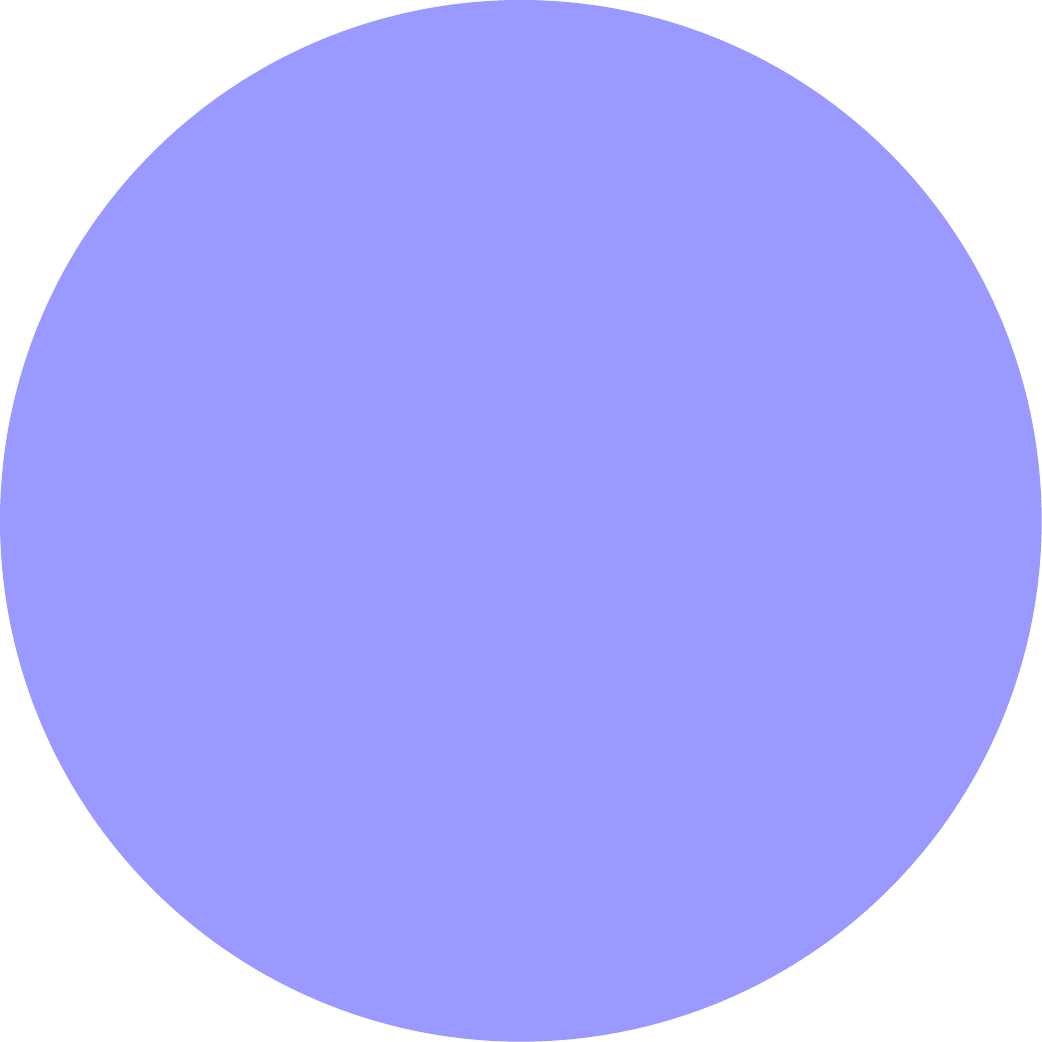}&
\text{Hypotenuse is composite}
\eeann \vskip.95in \hskip2in
\caption{A thin  orbit $\Or$ of Pythagorean triples, sifted by hypotenuse, $F(x,y,z)=z$.
 The darker points denote those triples whose hypotenuse is prime.
}
\label{fig0}
\end{figure}

\begin{rmk}
The weaker condition of {\it primitivity} requires the above for $q$ prime. See \cite[\S2]{BourgainGamburdSarnak2008} for an example of $(\Or, F )$ which is primitive but not 
strongly primitive. 
\end{rmk}

To present a concrete number $R(\Or, F)$, we will consider the quadratic form
$$
Q({\bx})=x^{2}+y^{2}-z^{2}.
$$
Hence a non-zero vector ${\bx}\in \z^3$ is a Pythagorean triple if $Q(\bx)=0$.
Let
 $G=\SO
 _{Q}(\R)
 $ be the  special orthogonal group preserving $Q$ with real entries.
For a discrete subgroup $\G$ of $G$, the critical exponent $0\le \delta_\G\le 1$
of $\G$ is defined to be the abscissa of convergence of the
Poincare series:
$$L_\G(s):=\sum_{\gamma\in \G} \|\gamma\|^{-s}$$
for any norm $\|\cdot \|$ on the vector space $\op{M}_3(\br)$ of $3\times 3$ matrices.
We remark that $\G$ is non-elementary if and only if $\delta_\G>0$. Moreover if $\G$ is finitely-generated, then
 $\G$ is of finite co-volume in $G$ if and only if $\delta_\G=1$ \cite{Patterson1976}.

The detailed statement of our main result is given in
Theorem \ref{thmMain}. The following is a special case:
\begin{thm}\label{thm:intro} Let $\G <\SO_Q(\z)$ be a finitely generated subgroup and
set  $$\Or:=(3,4,5)\G. $$
Let the polynomial $F$ be one of
 \begin{equation*}
 \begin{cases}
 \text{the hypotenuse: } & F_{\cH}({\bx}):=z;\\
\text{the ``area'': } & F_{\cA}({\bx}):=\frac 1{12} xy;
\\
\text{the product of coordinates : }& F_{\cC}({\bx}):=\frac 1{60} xyz .\end{cases}
 \end{equation*}
We assume that the pair $(\Or,F)$ is strongly primitive and that
\begin{equation*}
\gd
>
\begin{cases}
0.9992 &\text{
if $F=F_{\cH}$;}\\
0.99995 &\text{
if $F=F_{\cA}$;}
\\
0.99677
&\text{
if $F=F_{\cC}$.
}
\end{cases}\end{equation*}

Then the following hold:
\begin{enumerate}
\item
For infinitely many ${\bx}\in\Or$, the integer $F({\bx})$ has at most $R=R(\Or, F)$ prime factors, where
\begin{equation*}
R
=
\begin{cases}
14 &
\text{if $F=F_{\cH}$};\\
25 & \text{
if
$F=F_{\cA}$;}\\
29 &\text{
if
$F=F_{\cC}$.}\end{cases}\end{equation*}

\item
We have
$$
\#\{ {\bx}\in \Or: \|x\|<T,
\text{ $F({\bx})$ has at most $R(\Or, F)$ prime factors}\}
\asymp {T^{\delta_\G}\over (\log T)^{\gk}}
,
\footnote{
Recall that  $f\asymp g$ means that $ c^{-1}\cdot f(T)\le  g(T)\le c \cdot f(T)$ for some $c>1$ and for all $T>1$.
}
$$
where $\|\cdot \|$ is any norm on $\br^3$ and
the sieve dimension $\gk$ is
\begin{equation*}
\gk
=
\begin{cases}
1
&\text{
if
$F=F_{\cH}$;}
\\
4
&\text{
if
$F=F_{\cA}$;}\\ 
5&
\text{
if
$F=F_{\cC}$.}
\end{cases}\end{equation*}
%
%
In particular, the set of ${\bx}\in \Or$ such that $F({\bx})$ has at most
 $R(\Or, F)$ prime factors is Zariski dense in the cone $Q=0$.
\end{enumerate}
\end{thm}

\begin{rmk}
{\rm 
The functions $F_{\cA}$ and $F_{\cC}$
satisfy $F(3,4,5)=1$; hence
the pair $(\Or,F)$ is
strongly primitive regardless of the choice of the group $\G$.
For the hypotenuse, $F_{\cH}$,
one must check, given $\G$, that the pair $(\Or,F)$ is strongly primitive.
}
\end{rmk}

\begin{rmk}{\rm
The above theorem was proved in \cite{Kontorovich2009} assuming that
$\G$ contains a non-trivial (parabolic) stabilizer of $(3,4,5)$. In this case, the orbit $\Or$ contains an injection of affine space, and hence standard
sieve methods \cite{Iwaniec1978}
 also produce integral points with few prime factors.
Some of the most interesting cases which cannot be dealt with using  standard methods and
are now covered by our results are the so-called Schottky groups; these are
groups generated by finitely many hyperbolic elements.
}
\end{rmk}

\subsection{A Counting theorem} In order to sieve almost primes in a given orbit,
 one must know how to count points on such orbits, which we obtain without assuming the arithmetic condition
 on $\G$.


\begin{thm}\label{thm:count1}
Let $Q$ be
any ternary indefinite  quadratic form, 
$G=\SO_{Q}(\R)$, and
$\G<G$ a finitely generated discrete subgroup with $\gd_{\G} >1/2$.
Let ${\bx}_{0}\in\R^{3}$
be a non-zero vector lying on the cone $Q=0$
such that the orbit $\Or={\bx}_{0}\cdot\G$ is discrete.

Then
there exist a constant $c_0>0$ 
and some $\gz>0$
 such that
as $T\to\infty$, 
$$
\# \{{\bx}\in \Or: \| {\bx} \|<T\} =  c_0 \cdot T^{\gd}
+O(T^{\gd-\gz})
. $$
 The norm $\|\cdot\|$ above is 
Euclidean
.
\end{thm}

\begin{rmk}
{\rm
Let $N_0$ denote the (unipotent) stabilizer of ${\bx}_{0}$ in $G$.
The Theorem \ref{thm:count1} was proved in \cite{Kontorovich2009} under the further
assumption
that $\G\cap N_0$ is a lattice in $N_0$.
}
\end{rmk}

\subsection{Expanding Closed Horocycles}
The
main difference between this paper and \cite{Kontorovich2009} is the method  used to establish counting theorems
such as 
Theorem \ref{thm:count1}.
While \cite{Kontorovich2009} uses abstract operator theory, in the present work
we prove the effective equidistribution of expanding
closed horocycles
on a  
hyperbolic
surface $X$, allowing not only $X$ to have infinite volume, but also allowing the closed horocycle to be infinite in length.


Let
$G=\SL_2(\R)$ and write the Iwasawa decomposition $G=NAK$ with
\be\label{eq:NAdef}
N
=
\left\{
n_{x}=\begin{pmatrix} 1 & x\\ 0 &1\end{pmatrix}
:
x\in\R
\right\}
,
\quad
A=
\left\{
a_{y}=\begin{pmatrix} \sqrt y & 0\\ 0 &1/\sqrt y\end{pmatrix}
:
y>0
\right\}
,
\ee
and $K=\SO_2(\br)$.

We use the upper half plane $\bH=\{z=x+iy: y>0\}$ as a model for the hyperbolic plane
with the metric $\frac{\sqrt{dx^2+dy^2}}{y}$.
The group $G$ acts on $\bH$ by
fractional linear transformations which give arise all orientation preserving isometries of $\bH$:
$$\begin{pmatrix} a&b\\ c&d\end{pmatrix} z=\frac{az+b}{cz+d}$$
for $\op{Im}(z)>0$.
We compute:
$$n_xa_y(i)=x+iy .$$

Let $\G<G$ be a finitely generated discrete subgroup with $\delta_\G>1/2$.
 Assume that the horocycle $(\G\cap N)\bk  N$ is
 closed in $X:=\G\bk G$, or equivalently
 the image of $N(i)=\{x+i: x\in\br\}$ is closed in $\G\bk \bH$ under the canonical projection
  $\bH\to\G\bk \bH$.
    Geometrically, this is isomorphic to  either a line $\R$ or to a circle $\R/\Z$,
    depending on whether or not $\G\cap N$ is trivial.
%
%
We
push the closed horocycle $(N\cap \G)\bk N (i)$ in the
orthogonal
direction $a_{y}$, and
are concerned with its asymptotic distribution near the boundary, corresponding to  $y\to0$.

Let $X=\G\bk\bH$ and consider the Laplace operator $\gD=-y^2(\partial_{xx}+\partial_{yy})$.
By  Patterson \cite{Patterson1976} and  Lax-Phillips \cite{LaxPhillips1982},
the spectral resolution of  $\gD$ acting on $L^2(X)$ consists of  only finitely many eigenvalues in the interval $[0,1/4)$,  with the smallest given by
$\gl_{0}=\delta_\G(1-\delta_\G)$.
Denote the point spectrum below $1/4$ by 
$$
0\le\gl_0<\gl_1\le\dots\le \gl_k<1/4
.
$$
 Let $\phi_0,\dots,\phi_k$ be the corresponding eigenfunctions, normalized by $\|\phi_j\|_2=1$. Let $s_{j}>1/2$ satisfy $\gl_{j}=s_{j}(1-s_{j})$, $j=0,1,2,\dots,k$, so that $s_{0}=\gd$.

\begin{Thm}\label{equi}
Fix notation as above and assume  that $(\G\cap N)\bk N$ is closed.
 Then
for any $\psi\in C^\infty_c(\G\bk\bH)$,
$$
\int\limits_{n_x\in( N\cap \G)\ba N}
\hskip-.2in
\psi (x+iy)\; dx
=
\sum_{j=0}^k
\,
\langle \psi, \phi_j\rangle
\hskip-.2in
\int\limits_{n_x\in( N\cap \G)\ba N}
\hskip-.2in
\phi_j (x+iy)\; dx
 +O_{\vep}(
 y^
{
\frac 12 - \frac35(\gd-\frac 12)-\vep
}
 )
 ,
$$
as $y\to0$.
Here the implied constant depends
only on  a 
Sobolev norm 
of $\psi$, and on $\vep>0$ which is arbitrary.

Moreover, the integrals above converge, and satisfy
 $$
\int\limits_{n_x\in( N\cap \G)\ba N}
\hskip-.2in
\phi_j (
x+iy
)\; dx
\sim
c_{j}\cdot y^{1-s_{j}}
,
\qquad
\text{
as $y\to 0$,
}
$$
where $c_0>0$, and $c_1, \cdots, c_k \in \br$.
\end{Thm}

\begin{figure}
\includegraphics[width=.81in]{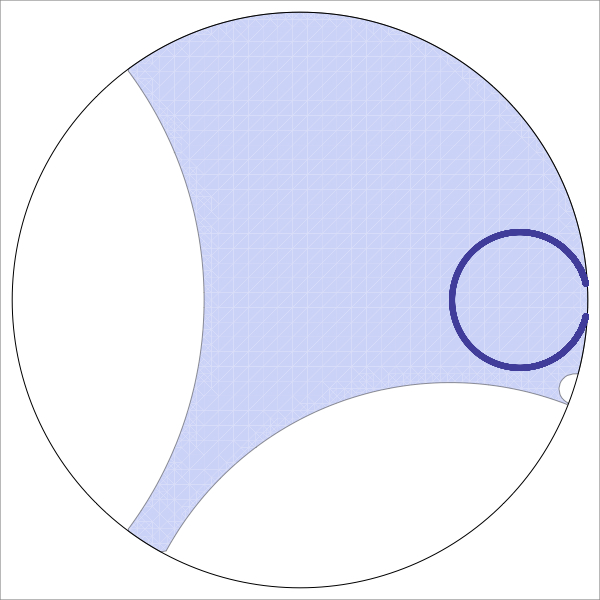}
\includegraphics[width=.81in]{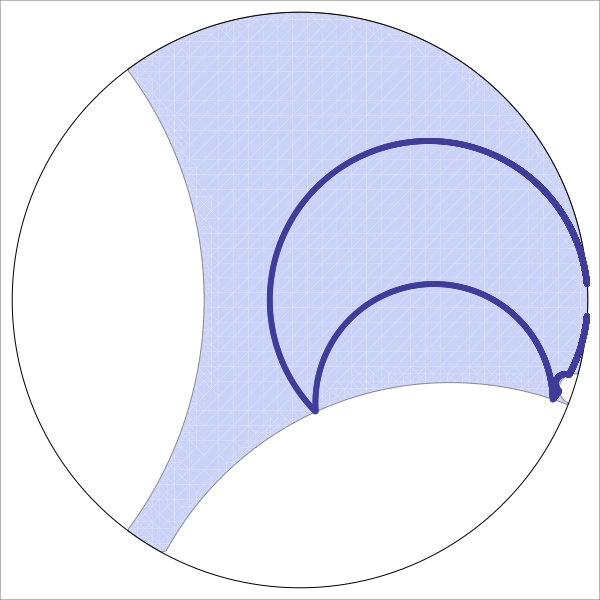}
\includegraphics[width=.81in]{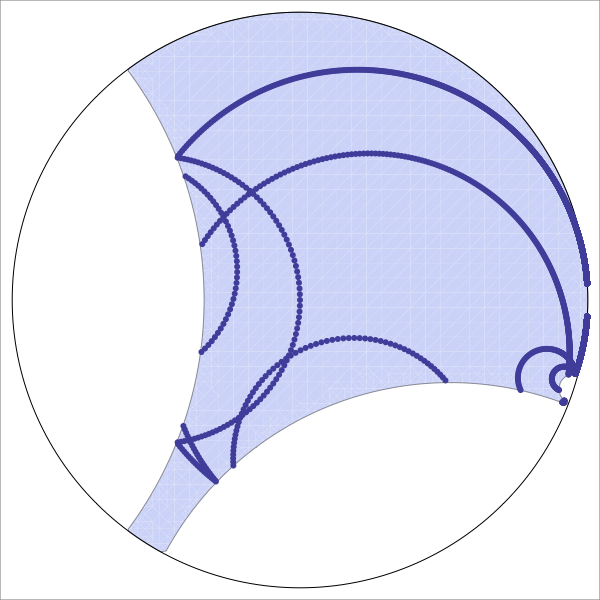}
\includegraphics[width=.81in]{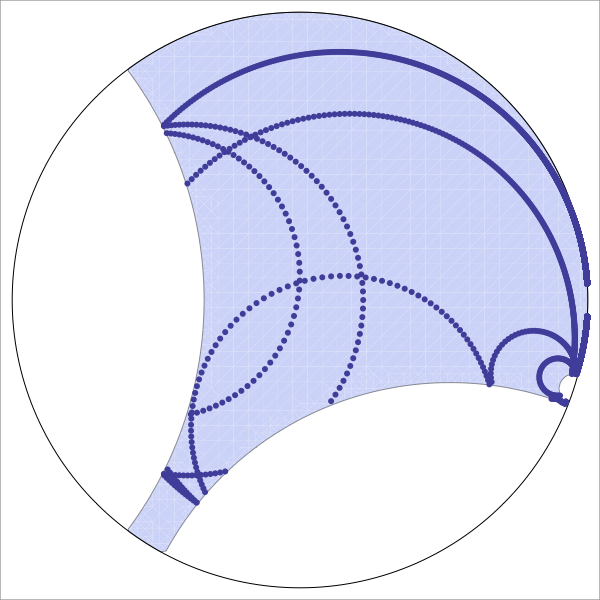}
\includegraphics[width=.81in]{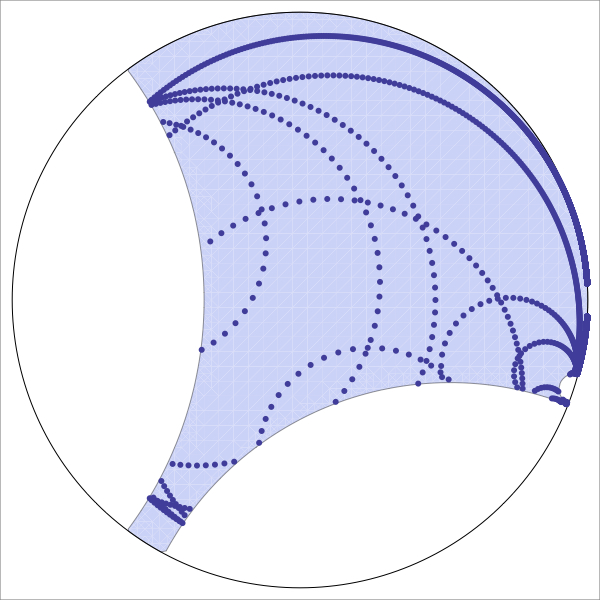}
\includegraphics[width=.81in]{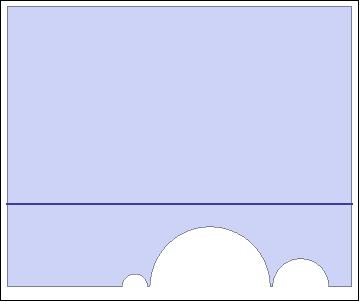}
\includegraphics[width=.81in]{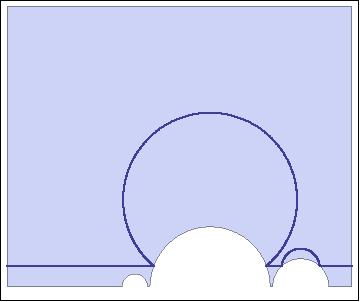}
\includegraphics[width=.81in]{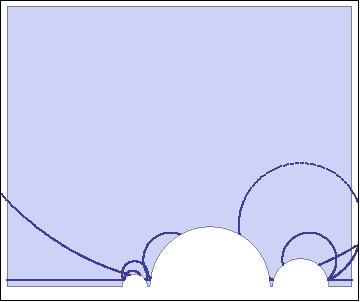}
\includegraphics[width=.81in]{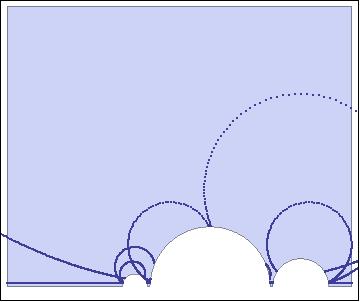}
\includegraphics[width=.81in]{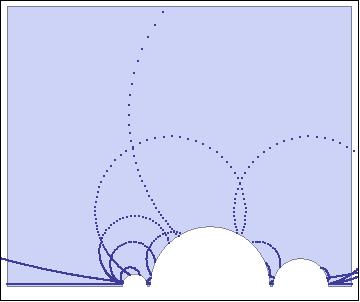}
\caption{Equidistribution of  expanding closed horocycles  on a Schottky domain in the disk $\bD$ and upper half plane $\bH$ models.}
\label{horo}\end{figure}

\begin{remark}
{\rm
If $\G$ is a lattice, then the closedness of $\Gamma\cap N\bk N$ implies
that $\Gamma\cap N\bk N$ is compact.
 In this case, Sarnak \cite{Sarnak1981} proved the above result allowing $\psi\in C^{\infty}_{c}(\G\bk G)$ (that is, not requiring $K$-fixed), and with
a best possible error term of
$$
 y^
{
\frac 12
}
$$
in place of our weaker bound
$$
 y^
{
{\frac 12} - \frac35(\gd-{\frac 12})
}
.
$$
}
\end{remark}

\subsection{Bounds for Automorphic Eigenfunctions}
The proof of
Theorem \ref{equi} requires control over the 
integrals of the eigenfunctions $\phi_{j}$, which {\it a priori} are only square-integrable.
For the base eigenfunction, one has extra structure coming from Patterson
theory \cite{Patterson1976} which makes this control possible. But for the other eigenfunctions,
this analysis fails. Nevertheless, the problem of obtaining such control was solved in the
first-named author's thesis \cite{MyThesis}.
The statement is the following (see the Appendix as well).

\begin{thm}[\cite{MyThesis}]\label{thm:hors}
Fix notation as in Theorem \ref{equi}, and assume that the closed horocycle $(N\cap\G)\bk N$ is
infinite. Let $\phi_{j}\in L^{2}(\G\bk \bH)$ be an eigenfunction of eigenvalue $\gl_{j}=s_{j}(1-s_{j})<1/4$ with $s_{j}>1/2$.  Then
$$
\phi_{j}
(n_{x}a_{y})
\ll_{\phi_{j}}
\left(
{y\over
x^{2}+y^{2}
}
\right)
^{s_{j}}
,
$$
as $|x|\to\infty$ and $y\to0$.
\end{thm}


\subsection{Organization of the Paper}

In \S\ref{secBack} we give some background and
elaborate further on Theorem \ref{thm:intro}.
For the reader's convenience,  in the Appendix we reproduce
the proof of Theorem \ref{thm:hors} from \cite{MyThesis}, since this reference  is not readily available.
Equipped with such control, the proof of Theorem \ref{equi} follows with minor changes from the one given for one dimension higher in \cite{KontorovichOh2008}. 
We
sketch the argument in \S\ref{sec:eq}, 
and
use it to
prove
Theorem \ref{thm:count2}
in \S\ref{countSec}.
In \S\ref{sieving}, we verify the sieve axioms  in Theorem \ref{sieve} and conclude Theorem \ref{thmMain}. At the end of \S\ref{sieving}, we derive the explicit values of $R$,
in particular proving Theorem \ref{thm:intro}.

\ 

\subsection*{ Acknowledgments.}
The authors wish to express their gratitude to Peter Sarnak  for many helpful discussions.


\section{Background and More on Theorem \ref{thm:intro}}\label{secBack}
In this section, we elaborate on Theorem \ref{thm:intro}.
Let $Q$ be a ternary rational quadratic form which is isotropic over $\q$.
Let $\G <\SO_Q(\z)$ be a finitely generated subgroup
with $\delta_\G >1/2$. 


As $Q$ is isotropic over $\q$, we have a $\q$-rational
covering $\SL_2\to \SO_Q$.
 Therefore
we may assume without loss of generality that $\G$ is a finitely generated subgroup
of $\SL_2(\z)$.

\subsection{Uniform Spectral Gaps}\label{subsec:SpecGap}
For the application to sieving, Theorem \ref{thm:count1} described in the introduction is insufficient. One requires uniformity
along arithmetic progressions; hence we recall the notion of a spectral gap.
\\

Let $\G(q)$ denote the ``congruence'' subgroup of $\G$ of level $q$, 
$$
\G(q):=\{\g\in\G:\g\equiv I(q)\}.
$$
The inclusion of vector spaces
$$
L^{2}(\G\bk\bH)
\subset
L^{2}(\G(q)\bk\bH)
$$
induces the same inclusion on the spectral resolution of the Laplace operator:
$$
\Spec(\G\bk\bH)
\subset
\Spec(\G(q)\bk\bH)
.
$$

\begin{Def}
The {\it new spectrum}
$$
\Spec_{new}(\G(q)\bk\bH)
$$
at level $q$ is defined to be the set of eigenvalues
below
$
1/4
$
which are in
$
\Spec(\G(q)\bk\bH)
$
but  not in
$
\Spec(\G\bk\bH)
.
$
\end{Def}

\begin{Def}\label{fBdef}
{\rm
A number $\gt$ in the interval $1/2<\gt<\gd$ is called a {\it spectral gap} for $\G$
if there exists a {\it ramification number} $\fB\ge1$ such that  for any square-free
$$
q=q'q''
\quad
\text{ with }
\quad
q'\mid\fB
\text{ and }(q'',\fB)=1,
$$
we have
$$
\Spec(\G(q)\bk\boldH)_{new}
\cap
(0,\gt(1-\gt))
\quad
\subset
\quad
\Spec(\G(q')\bk\boldH)_{new}.
$$
}
\end{Def}

That is, the eigenvalues below $\gt(1-\gt)$ which are new for $\G(q)$ are coming from the ``bad'' part $q'$ of $q$. As $\fB$ is a fixed integer depending only on $\G$, there are only finitely many possibilities for its divisors $q'$.


Collecting the results in \cite{BourgainGamburd2007,BourgainGamburdSarnak2008} and their extension from prime to square-free of \cite{Gamburd2002} we have:

\begin{thm}[\cite{Gamburd2002,BourgainGamburd2007,BourgainGamburdSarnak2008}]\label{BGSspec}
\

\begin{enumerate}
\item
For any
finitely generated $\G<\SO_Q(\z)$ with critical exponent $\gd>1/2$,
there exists some spectral gap
$$
1/2<\gt<\gd.
$$

\item If $\gd>5/6$ then  $\gt=5/6$ holds.
\end{enumerate}
\end{thm}

\subsection{Counting with Weights uniformly in Level}
Allowing some ``smoothing'', one can count uniformly in cosets of orbits of level $q$ with explicit error terms.
We fix a non-zero vector ${\bx}_0\in \z^3$ with $Q({\bx}_0)=0$
and set $$\Or={\bx}_0\G .$$ 
We denote by $N_0$
the stabilizer subgroup of ${\bx}_0$ in $G:=\SL_2(\br)$.
Then $$N_0=g_0Ng_0^{-1}$$ for some $g_0\in \SL_2(\q)$,
where $N$ denotes the upper triangular subgroup of $G$.

Set $K_0:=g_0\SO_2(\br) g_0^{-1}$
and choose $\eta>0$ so that a $K_0$-invariant
 $\eta$ neighborhood $U_\eta$ of $e$ in $G$ injects to $\G\bk G$.
Let $\psi:=\psi_\eta$ be a non-negative smooth $K_0$-invariant function
 on $G$ supported in $U_\eta$ with $\int \psi dg=1$.


Denote by $B_T$ a $K_0$-invariant norm ball in $\br^3$ about the origin
 with radius $T$.
\begin{Def} \label{xiDef}{\rm
The weight $\xi_T: \br^3\to \br_{\ge 0}$ is defined as follows:
$$ \xi_T({\bx})=\int_G \chi_T({\bx}g)\psi(g) dg $$
where $\chi_T$ denotes the
 characteristic function of $B_T$.}\end{Def}

The sum  of $\xi_T$ over $\Or$ is precisely a smoothed count for $\#\Or\cap B_T$ satisfying:
$$\sum_{{\bx}\in\Or}\xi_T({\bx})\asymp \# \Or\cap B_T .$$

\begin{thm}\label{thm:count2}
Let $\gt$ be the spectral gap for $\G$.
\begin{enumerate}
\item
As $T\to\infty$,
$$
\Xi(T)
:=
\sum_{{\bx}\in{\Or}} \xi_{T}({\bx})
\sim
c\cdot T^{\gd},
$$
for some $c>0$. \\

\item For square-free $q$, write $q=q'q''$ with $q'\mid \fB$ and $(q'',\fB)=1$. Let $\G_{1}(q)$ be any group satisfying
$$
\G(q)\subset\G_{1}(q)\subset\G
.
$$
Let $N_0$ be the stabilizer of ${\bx}_{0}$ in $G$, and assume that
$$
\G_{1}(q)\cap N_0 = \G\cap N_0.
$$
Fix any $\g_{0}\in\G$ and 
$\vep>0$.
Then
 as $T\to\infty$,
 \beann
\sum_{{\bx}\in{\bx}_{0} \G_{1}(q)}
\xi_{T}({\bx} \g_{0})
&=&
{1\over [\G:\G_{1}(q)]}
\cdot
\bigg(
\Xi(T)
+
\cE(T,q',\g_{0})
\bigg)
\\
&&
+
O_{\vep}
\bigg(
T^{\gt+\vep}
+
T^{{\frac 12}+\frac3{5}(\gd-{\frac 12})+\vep}
\bigg)
,
\eeann
 where the implied constant does not depend on $q$ or $\g_{0}$. Here
the error term satisfies
 $$
 \cE(T,q',[\g_{0}])
\ll
T^{\gd-\zeta}
$$
for some fixed $\zeta>0$, does not depend on $q''$, and depends only on the class $[\g_{0}]$ in $\G_{1}(q')\bk\G$.
\end{enumerate}

\end{thm}

\begin{rmk}
{\rm
Assuming
that $\G\cap N_0$ is a lattice in $N_0$,
 \cite{Kontorovich2009}
gives the above uniform count with the last error term
$$
T^{{\frac 12}+\frac3{5}(\gd-{\frac 12})+\vep}
$$
replaced by a best possible error of
$$
T^{{\frac 12}}\log T
.
$$
}
\end{rmk}








\

\subsection{Zariski Density of Orbits of Pythagorean Triples}
\

For simplicity, we will use the notation
$\cP(R)$ to denote the set of all integers having at most $R$ prime divisors.

Let ${\bx}_{0}\in\Z^{3}$ be a non-zero Pythagorean triple on the cone
$$
Q({\bx})=x^{2}+y^{2}-z^{2}
=0 $$
and $\G<\SO_Q(\Z)$ a non-elementary finitely generated subgroup.
Set
$$\Or:={\bx}_{0}\cdot \G.$$

Given a polynomial $F$ which is integral on $\Or$, our goal
is to find ``small''
 values for $R=R(\Or,F)$,
 for which $F$ ``often'' has at most $R$ prime factors.

In fact, when studying such thin orbits, the correct  notion of ``often''  is not ``infinitely often'', but instead one should require Zariski density. That is, the set of ${\bx}\in\Or$ for which $F({\bx})\in\cP(R)$
should not lie on a proper subvariety of the smallest variety 
containing $\Or$. We illustrate this condition with the following examples.

\subsubsection{Example I: Area}
\

Recall that given any integral Pythagorean triple ${\bx}=(x,y,z)$ which is also primitive (that is, there is no common divisor of $x$, $y$ and $z$), there exist coprime integers $u$ and $v$ of opposite parity (one even, one odd) such that, possibly after switching or negating $x$ and $y$, we have  the ancient parametrization 
\begin{equation*}\label{param}
x=u^2-v^2,\quad y=2u\,v,\quad z=u^2+v^2.
\end{equation*}
In fact, this is just a restatement of the group homomorphism $\SL_2(\br)\to \SO(2,1)$ given by
 \begin{equation*}
\mattwo abcd \mapsto
 \begin{pmatrix}
{a^2-b^2-c^2+d^2\over 2}& {ac-bd} &{a^2-b^2+c^2-d^2\over 2}\\
{ab-cd}& {bc+ad}& {ab+cd}\\
{a^2+b^2-c^2-d^2\over 2}& {ac+bd} &{a^2+b^2+c^2+d^2\over 2}
\end{pmatrix}
\end{equation*}
where $\SL(2,\R)$ acts on $(u,v)$ and $\SO(2,1)$ acts on $(x,y,z)$.

Consider the ``area'' ${\frac 12} xy$ of the triple ${\bx}$ (which may be negative).
It is elementary that the area is always divisible by $6$, so
the function
\be\label{fdef}
F_{\cA}({\bx})
:=
{1\over 12}xy
=
{1\over 6}(u+v)(u-v)u\,v
\ee
is integer-valued on $\Or$.

\begin{rmk}
{\rm
As above, we insist that the polynomial $F$ is integral on $\Or$, but it need not necessarily have integer coefficients.
}

\end{rmk}

As \eqref{fdef} has four irreducible components, it is easy to show that there are only finitely many triples ${\bx}$ for which $F_{\cA}({\bx})\in\cP({2})$, that  is, the product of at most two primes. Restricting to a subvariety such as
$$
u=v+1,
$$
it follows conjecturally from the Hardy-Littlewood $k$-tuple conjectures  \cite{HardyLittlewood1922} that
$$
\frac1{12} xy=
{1\over 6}(u+v)(u-v)u\,v
=
{1\over 6}(2v+1)\cdot 1\cdot (v+1)\cdot v
$$
 will be the product of three primes for infinitely many $v$.

Since the set of triples generated in this way lies on a subvariety, it is not Zariski dense. On the other hand,
it was recognized in \cite{BourgainGamburdSarnak2008} that the recent work of Green and Tao \cite{GreenTao2006} proves the infinitude and Zariski density of
the set of all primitive Pythagorean triples ${\bx}$ for which $F_{\cA}({\bx})\in\cP({4})$, that is, has at most four prime factors.

\begin{rmk}
{
\rm
The results of Green-Tao do not apply to thin orbits, and neither do the conjectures of Hardy-Littlewood.
Indeed, 
we conjecture that
that if $\Or$ is thin
and $\G$ has no unipotent elements (which would furnish an affine injection into $\Or$),
 then there are only {\it finitely-many} points ${\bx}$ for which $F_{\cA}({\bx})\in\cP({3})$! 
On the other hand, allowing $4$ primes should lead to a Zariski dense set of triples ${\bx}$. Below, we exhibit certain thin orbits for which there is a Zariski dense set of ${\bx}$ with $F_{\cA}({\bx})\in\cP(25)$. 
}
\end{rmk}

\begin{rmk}
{
\rm
The critical number, $4,$ of prime factors above is related to the {\it sieve dimension} for this pair $(\Or, F)$. We return to this issue shortly, cf. Remark \ref{rmk:LiuS}.
}
\end{rmk}

\subsubsection{Example II: Product of Coordinates}

Consider now the product of coordinates $xyz$ for triples 
${\bx}
\in\Or$. It is elementary that $xyz$ is divisible by $60$, so
the function
\be\label{f2def}
F_{\cC}({\bx})
:=
{1\over 60}xyz
=
{1\over 30}(u+v)(u-v)u\,v(u^2+v^2)
\ee
is integer-valued.

\begin{figure}
\hskip-3.5in
\includegraphics[width=2.5in]{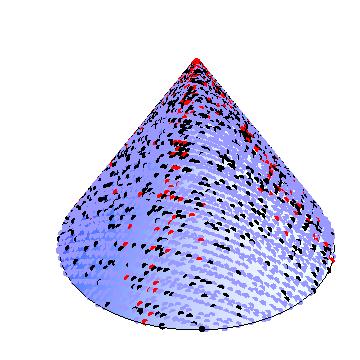}
\vskip-2.25in \beann\hskip1.7in &
\includegraphics[width=.1in]{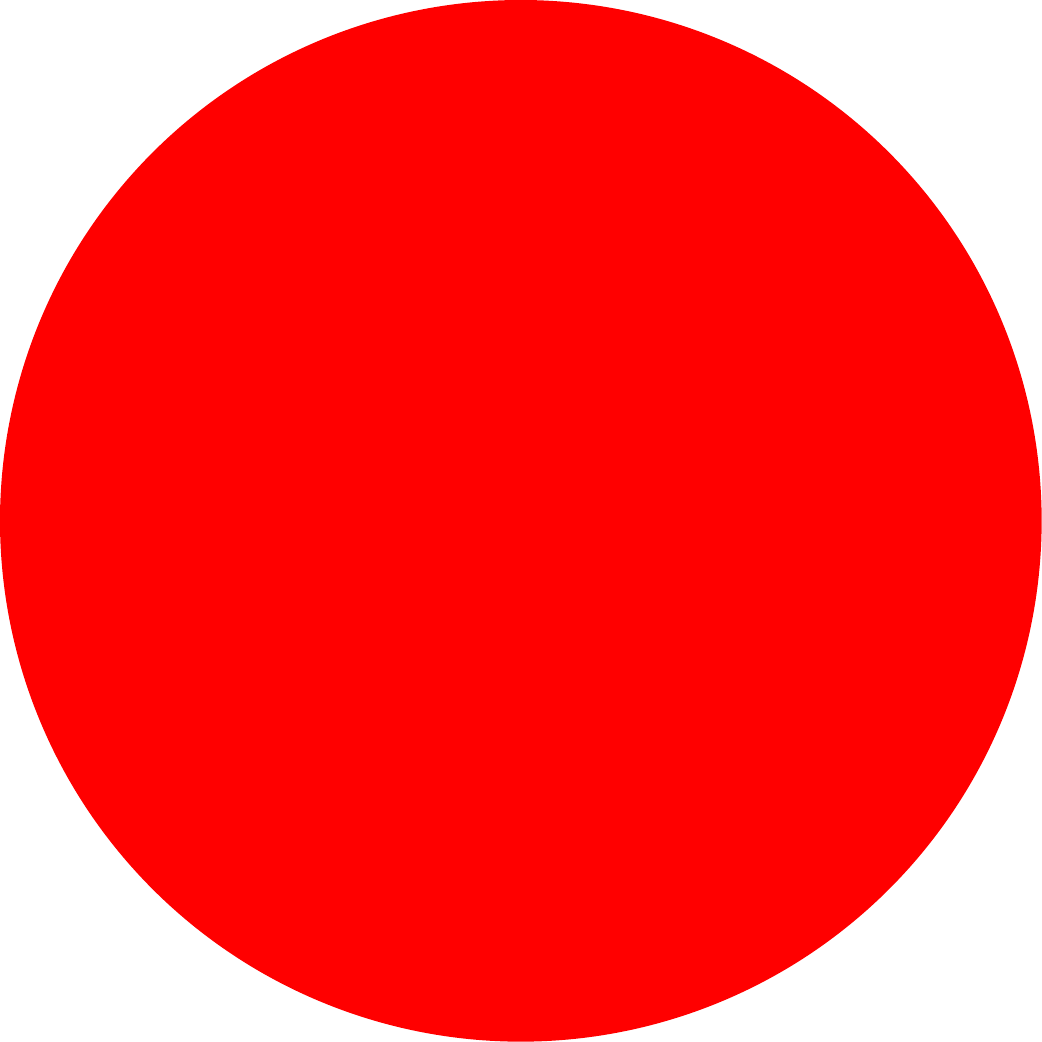} &
\text{${1\over 60}xyz$ has at most four prime factors} \\ &
\includegraphics[width=.1in]{black} &
\text{${1\over 60}xyz$ has exactly five prime factors}
\\ & \includegraphics[width=.1in]{light}
& \text{${1\over 60}xyz$ has  six or more  prime factors}
\eeann \vskip.65in \hskip2in
\caption{ The full orbit of all primitive Pythagorean triples. }\vskip.15in
\label{full} \end{figure}

\begin{figure} \hskip-3.5in
\includegraphics[width=2.5in]{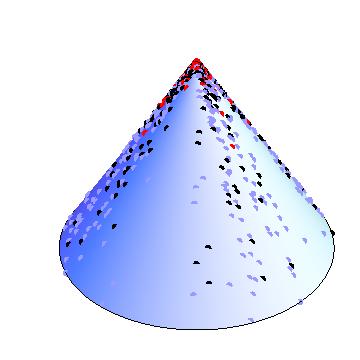}
\vskip-2.25in \beann\hskip1.7in &
\includegraphics[width=.1in]{red} &
\text{${1\over 60}xyz$ has at most four prime factors} \\ &
\includegraphics[width=.1in]{black} &
\text{${1\over 60}xyz$ has exactly five prime factors}
\\ &\includegraphics[width=.1in]{light} &
\text{${1\over 60}xyz$ has  six or more  prime factors}
\eeann \vskip.65in \hskip2in
\caption{A thin  orbit $\Or$ of Pythagorean triples.}
\vskip.15in \label{fig1} \end{figure}


Now we note that as \eqref{f2def} has five irreducible components (and 
the {\it sieve dimension}  is five). Therefore  there are only finitely many triples ${\bx}$ for which $F_{\cC}({\bx})\in\cP(3)$. 
Restricting to a subvariety such as
$$
u=v+3,
$$
it follows conjecturally from Schinzel's Hypothesis H \cite{SchinzelSierpinski1958} that

$$
{1\over 30}(u+v)(u-v)u\,v(u^2+v^2)
=
{1\over 30}
(2v+3)
\cdot
(
3
)
\cdot
(v+3)
\cdot
v
\cdot
(2v^2+6v+9)
$$
 will be the product of four primes for infinitely many $v$. Again, this set is not Zariski dense in the cone.
 See
Figure \ref{full},
  where it is clear that such points frolic near the $x$ or $y$ axes. 

On the other hand, it is a folklore conjecture that the set of triples ${\bx}$ for which $F_{\cC}({\bx})\in\cP(5)$ 
spreads
out in every direction.
For the full orbit of all primitive Pythagorean triples
 (rather than a thin one),
the best known bound for the number of prime factors in $F_{\cC}({\bx})$
follows from the Diamond-Halberstam-Richert sieve
\cite{DiamondHalberstamRichert1988,DiamondHalberstam1997}. Their work shows that
$F_{\cC}({\bx})\in\cP(17)$ 
infinitely often.%
\footnote{
In fact they restrict to  a subvariety in deriving the numer $17$.
}

Again, when the orbit $\Or$ is thin without affine injections,
we conjecture
that there will be only finitely many ${\bx}$ for which $F_{\cC}({\bx})\in\cP(4)$ 
whereas five factors will be Zariski dense. Compare
Figure \ref{full} to Figure \ref{fig1}.
We will exhibit certain thin orbits
for which there is a Zariski dense set of ${\bx}$ with $F_{\cC}({\bx})\in\cP(29)$. 

\begin{rmk}\label{rmk:LiuS}
{\rm
Sieve dimension is not merely a function of the polynomial $F$ but really depends on the pair $(\Or,F)$. In the related recent work   \cite{LiuSarnak2007},  Liu and Sarnak
consider $F_{\cC}({\bx})=xyz$, where the orbit $\Or={\bx}_{0}\cdot \G$ is generated from a point ${\bx}_{0}\in\Z^{3}$ on a one- or two-sheeted hyperboloid $Q({\bx})=t$, where $t\neq0$ and $Q$ is an indefinite integral ternary quadratic form which is anisotropic. Then the spin group of $G=\SO_{Q}(\R)$ consists of the elements of norm one in a quaternion division algebra over $\Q$, and $\G$ is the set of {\it all} such integral elements.

In particular, their orbit is full, whereas the focus of this paper is on thin orbits. A common feature, though, is that there do not exist non-constant polynomial parametrizations of points in $\Or$ (in our case this corresponds precisely to $N_0\cap \G$ being trivial).

The sieve dimension for Liu-Sarnak's pair $(\Or,F_{\cC})$ is $3$ (whereas in our case, the same function $F=F_{\cC}$ has sieve dimension $5$), and they prove the Zariski density of the set of points ${\bx}\in\Or$ for which  $F_{\cC}({\bx})$ is in $\cP(26)$.
The precise definition of sieve dimension is given in Definition \ref{sieveD}.
}
\end{rmk}



\subsection{The Diamond-Halberstam-Richert Weighted Sieve}

\

Let $\cA
=\{a_n\}$ be a
sequence of non-negative real numbers,
all but finitely many  of which
are zero.

For $R\ge1$,
let
$\cP(R)$ be the set of all integers having at most $R$ prime divisors counted with multiplicity.
Let
$\cW=\{\gw(n)\}$ denote a certain sequence of weights supported on square-free numbers 
satisfying
\be\label{gwdef}
\gw(n)\le 0
\quad
\text{  if }
\quad
n\notin \cP(R)
,
\ee
and being of convolution type, that is
\be\label{convType}
\gw(n)=\sum_{d|n}\gw(d).
\ee

Let
$$
S(\cA\cW):= \sum_{n} a_n\,\gw(n)
.
$$
If we can
construct a suitable sequence $\cW$
with
a good lower bound estimate for $S(\cA\cW)$, 
then we can conclude by virtue of \eqref{gwdef}  that there are elements  $a_n\in\cA$ with $n\in\cP(R)$, that is, having at most $R$ prime divisors.


Moreover, by \eqref{convType} we can extract estimates for $S(\cA\cW)$ from knowledge of the distribution of $\cA$ along certain arithmetic progressions. For $q\ge1$ a square-free integer, let
$$
\cA_q:= \{a_n\in\cA:n\equiv0(q)\} \quad \text{and}\quad
|\cA_q|:=\sum_{n\equiv0(q)} a_n,
$$
so that
$$
S(\cA\cW) = \sum_q \gw(q) |\cA_q|.
$$
Assume there exists an approximation $\cX$ to $|\cA|:=\sum_n a_n$
and a non-negative multiplicative function $g(q)$ so that $g(q)\cX$ is an approximation to $|\cA_q|%
$. Assume that $g(1)=1$,
 $g(q)\in[0,1)$ for $q>1$,
  and that for constants $K\ge2$, $\gk\ge1$
we have the local density bound
\be\label{locDens}
\prod_{{z_1\le p \le z}
} (1-g(p))^{-1} \le  \left(\frac{\log z}{\log z_1}\right)^\gk\left(1+{K\over\log z_1}\right)
\ee
for any $2\le z_1<z$.
\begin{Def}\label{sieveD}
{\rm
The number $\gk$ appearing in \eqref{locDens} is called the {\it sieve dimension}.
(Note that it is not unique, as any larger value also satisfies \eqref{locDens}; in practice one typically takes the least allowable value.)
}
\end{Def}
We require that the remainder terms
$$
r_q := |\cA_q| - g(q) \cX
$$
be small an average, that is, for some constants $\tau\in(0,1)$ and $A\ge1$,%
\footnote{
Recall that $\nu(q)$ denotes the number of prime factors of $q$.
}
\be\label{taudef}
\sum_{q<\cX ^\tau(\log \cX)^{-A}
\atop q\text{ squarefree}
}
4^{\nu(q)} |r_q| \ll {\cX\over \log^{\gk+1}\cX}
.
\ee

Finally, we introduce a parameter $\mu$ which controls the number of terms in $\cA$ which are non-zero. Precisely, we require that
\be\label{mudef}
\max\{n\ge1: a_n\neq0\}\le \cX^{\tau\mu}.
\ee
We now state
\begin{Thm}[\cite{DiamondHalberstamRichert1988,DiamondHalberstam1997}]\label{sieve}
Let $\cA$,
$z$,
$\cX$, $g$, $\gk$, $\mu$ and $\tau$ be as described above.
\begin{enumerate}
\item
Let $\gs_\gk(u)$ be the continuous solution of the differential-difference problem:
\be
\begin{cases}
{u^{-\gk}\gs(u)=A_\gk^{-1},} & \text{for $0<u\le2$, $A_\gk = (2e^\g)^\gk\G(\gk+1)$,}\\
{(u^{-\gk}\gs(u))'=-\gk u^{-\gk-1} \gs(u-2),}&\text{for $u>2$,}
\end{cases}\ee
where $\g$ is the Euler constant. Then there exist two numbers $\ga_\gk$ and $\gb_\gk$ satisfying $\ga_\gk\ge\gb_\gk\ge2$ such that the following simultaneous differential-difference system has continuous solutions $F_\gk(u)$ and $f_\gk(u)$ which satisfy
$$
F_\gk(u)=1+O(e^{-u}), \quad f_\gk(u)=1+O(e^{-u}),
$$
 and $F_\gk$ (resp. $f_\gk$) decreases (resp. increases) monotonically towards $1$ as $u\to\infty$:
\be
\begin{cases}
{F(u)=1/\gs_\gk(u),} & \text{for $0<u\le\ga_\gk,$}\\
{f(u)=0,} & \text{for $0<u\le\gb_\gk,$}\\
{(u^\gk F(u))'=\gk u^{\gk-1}f(u-1),} &\text{for $u>\ga_\gk,$}\\
{(u^\gk f(u))'=\gk u^{\gk-1}F(u-1),} &\text {for $u>\ga_\gk.$}
\end{cases}\ee

\item
For any two real numbers $u$ and $v$ with
$$
 \tau^{-1}<u\le v, \gb_\gk<\tau v,
$$
there exist weights $\cW=\{\gw(n)\}$ such that
$$
S(\cA\cW) \gg \cX \prod_{p<X^{1/v}}(1-g(p))
,
$$
provided that
\be\label{Ris}
R>\tau \mu u - 1 +
 {\gk \over f_\gk(\tau v)} \int_1^{v/u}F_\gk(\tau v-s)\left(1-\frac uv s\right){ds\over s}
 .
\ee
\end{enumerate}
\end{Thm}

\subsection{Statement of the Main Theorem}

The following is the main result of this paper.

\begin{thm}\label{thmMain} Let $Q({\bx})=x^2+y^2-z^2$, ${\bx}_{0}\in \z^3$ a non-zero vector with $Q({\bx}_0)=0$,
and $\G <\SO_Q(\z)$  a finitely generated subgroup with $\gd >1/2$.
Denote by $\gt$ the spectral gap of $\G$.
 Let $F$ be a polynomial which is integral on $\Or$, such that the pair
 $(\Or,F)$ is strongly primitive.

Then the following hold:
\begin{enumerate}
\item
There are infinitely many ${\bx}\in\Or$ such that $F({\bx})$ has at most $R$ prime factors,
where $R(\Or,F)$ is given by \eqref{Ris} with
\be\label{mutau}
\tau
<
\min
\left(
{\gd-\gt
\over 2\gd}
,
{
\gd-{\frac 12}\over
5\gd
}
\right)
,
 \qquad\text{ and }\qquad
\mu
>
\max
\left(
{2\over
\gd-\gt}
,
{5\over
\gd-1/2}
\right)
.
\\
\ee
\item
Under the further assumption that $N_0\cap \G$ is a lattice in $N_0$ for $N_0=\op{Stab}_G({\bx}_0)$,
   the bounds in \eqref{mutau} improve to
\be\label{taumuLattice}
\tau
<
{\gd-\gt
\over 2\gd}
,
 \qquad\text{ and }\qquad
\mu
>
{2\over
\gd-\gt}
.
\ee
\item
 Denoting by $\gk$ be the sieve dimension of $(\Or,F)$,
$$\#\{{\bx}\in \Or: \|{\bx}\|<T, F({\bx})\in \cP(R)\}
\asymp
{T^{\gd}\over (\log T)^{\gk}}.
 $$
 In particular, this set is Zariski dense in the cone $Q=0$.
\end{enumerate}
\end{thm}

\subsection{Explicit Values of $R(\Or,F)$}

One must still work somewhat to obtain  actual values of $R$ from Theorem \ref{thmMain}.
We now
state the smallest values of $R(\Or,F)$
 which are achieved from Theorem \ref{thmMain}, at the expense of requiring the critical exponent $\gd$ to be close to $1$.\\

 The first statement we give is unconditional.
Assuming $\gd>5/6$
 and using Gamburd's spectral gap $\gt=5/6$ from Theorem \ref{BGSspec}, the bounds \eqref{mutau} and \eqref{taumuLattice} are equivalent. Hence assuming $N_0\cap\G$ is a lattice in $N_0$ and using the optimal error terms in \cite{Kontorovich2009} does not improve the final values of $R$.
 Said another way, the fact that our counting theorem is not optimal does not hurt the final values of $R$, unconditionally.

In the next three theorems, we keep the notation $Q, \Gamma, \Or$ from Theorem \ref{thmMain}. Let
$$
F_{\cH}({\bx})=z, \qquad
F_{\cA}({\bx})=\frac1{12}xy, \qquad
F_{\cC}({\bx})=\frac1{60}xyz,
$$
and assume the pair $(\Or,F)$ is strongly primitive.

The following is the same as Theorem \ref{thm:intro}:
\begin{thm}\label{numbs}
Let $\G$ have critical exponent
$$
\delta_{\cH}>0.9992
,\qquad
\delta_{\cA}>0.99995
,\qquad
\delta_{\cC}>0.99677
.
$$
Then the proportion of ${\bx}\in\Or$ with $F({\bx})\in\cP(R)$ with $\|{\bx}\|<T$ is
$$
\asymp
{1\over (\log T)^{\gk}}
,
$$
where
$$
\gk_{\cH}=1,\qquad
\gk_{\cA}=4,\qquad
\gk_{\cC}=5,
$$
and
$$
\boxed{
R_{\cH}=14,\qquad
R_{\cA}=25,\qquad
R_{\cC}=29.
}
$$
\end{thm}


We now observe the effect of the worse error term on the quality of $R(\Or,F)$, 
conditioned on an improved spectral gap.
%
Assuming $N_0\cap \G$ is a lattice in $N_0$, the counting theorem in \cite{Kontorovich2009} gives an optimal error term, which 
together with our sieve analysis gives 
the following.

\begin{thm}\label{numbs1}
 Assume that
 $N_0\cap \G$ is a lattice in $N_0$,
 and that the spectral gap
 $\gt$ can be arbitrarily close to $1/2$. Then if \\
$$
\delta_{\cH}>0.9265
,\qquad
\delta_{\cA}>0.98805
,\qquad
\delta_{\cC}>0.981675
,
$$
  the conclusion of Theorem \ref{numbs} holds with
$$
\boxed{
R_{\cH}=6,\qquad
R_{\cA}=14,\qquad
R_{\cC}=17.
}
$$
\end{thm}

Lastly, we demonstrate the values of $R$ which we can obtain without assuming
 that $\G\cap N_0$ is a lattice in $N_0$.

\begin{thm}\label{numbs2}
We make no assumptions on 
 $N_0\cap \G$, but assume that $\gt$ can be arbitrarily close to $1/2$. 
Then if
$$
\delta_{\cH}>0.991
,\qquad
\delta_{\cA}>0.97895
,\qquad
\delta_{\cC}>0.99905
,
$$
 the conclusion of Theorem \ref{numbs} holds with
$$
\boxed{
R_{\cH}=12,\qquad
R_{\cA}=23,\qquad
R_{\cC}=26.
}
$$
\end{thm}

Theorems \ref{numbs}, \ref{numbs1}, and \ref{numbs2} follow from Theorem \ref{thmMain} and Table \ref{FigTab}, the computation of which is discussed in
\S\ref{sec:valsR}.

\section{Equidistribution of Expanding Closed  Horocycles}\label{sec:eq}
Let $G=\SL(2,\R)$. We keep the notation for $N, A, K,  n_x, a_y$, etc., from
 \eqref{eq:NAdef}.
We have the Cartan decomposition $$G=KA^+K$$
where $A^+:=\{a_y: 0< y\le 1\}$, as well
as the Iwasawa decomposition $G=NAK$.

%
Let $\G<G$ be a discrete finitely generated subgroup with critical exponent $\gd>1/2$. 
Assume the horocycle $(N\cap\G)\bk N$ is closed in $\G\ba G$.
In this section, we prove
Theorem \ref{equi}.


%
%
%

\subsection{Automorphic Representations and Spectral Bounds}

%

Let $1/2< s<1$,
and consider the character $\chi_s$ on the upper-triangular subgroup $B:=NA$ of $G$
defined by $$\chi_s(na_y )= y^{s} $$
where $a_y=\text{diag}(\sqrt y, \sqrt y^{-1})$ is as before, and $n\in N$.

The unitarily induced
representation $(\pi_s:=\op{Ind}_{B}^G \chi_s, V_s)$
 admits a unique $K$-invariant unit vector,
say $v_s$.

By the theory of spherical functions,
$$
f_s(g):=\langle \pi_s(g)v_s, v_s\rangle =\int_K v_s(kg)dk
$$
is the unique bi $K$-invariant function
of $G$
with $f_s(e)=1$ and with $\mathcal C f_s = s (1-s) f_s$ where $\mathcal C$ is the Casimir operator
of $G$.
Moreover, there exist some $c_s>0$ and $\ga>0$ such that for all $y$ small
 $$f_s(a_y)=c_s
\cdot y^{1-s}  (1+O(y^{\ga }) )$$
by \cite[4.6]{GangolliVaradarajanbook}.




Since the Casimir operator is equal to the Laplace operator $\Delta$ on $K$-invariant functions,
this immediately implies the following.

\begin{Thm}\label{as}  Let $\phi_s\in L^2(\G\bk G)^K\cap C^\infty(\G\bk G)$
satisfy $\Delta \phi_s= s(1-s) \phi_s$ and $\|\phi_s\|_2=1$.
Then there exist $c_s>0$ and $ \ga>0$ such that for all 
$y\ll 1$,
$$\langle a_y\phi_s, \phi_s\rangle_{L^2(\G\bk G)}=c_s\cdot y^{1-s} (1+O(y^{\ga })) .$$
\end{Thm}

The irreducible unitary representations of $G$ with $K$-fixed vector consist of
principal series and the complimentary series.
We use the parametrization of
 $s\in\{1/2 +i\R \}\cup [1/2,1]$
so that the vertical line
$ 1/2 +i\R$
corresponds to the principal series and the complimentary series is parametrized by
$ 1/2<  s \le 1$ with
$s=1$ corresponding to the trivial representation.

Let $\{X_1,X_2,X_3\}$ denote an orthonormal basis of the (real) Lie algebra of $G$
with respect to an $\text{Ad}$-invariant scalar product.
For $f\in C^\infty (\G\bk G)\cap L^2(\G\bk G)$, we consider the following
 Sobolev norm $\mathcal S_m(f)$:
$$  \mathcal S_m(f)=\max \{\| X_{i_1}\cdots X_{i_m} (f) \|_2 :
1\le i_j\le 3 \}.
$$

The following is well-known (cf. \cite{
CowlingHaagerupHowe1988, 
Knapp1986}).


\begin{Prop}\label{decay}
Let $(\pi, V)$ be a representation of $G$
which does not weakly contain
 any  complementary series representation
$V_s$. 
Then  for any $\vep>0$,
and
any smooth vectors $v_1, v_2\in V$,
$$
|\langle \pi(a_y) v_1, v_2\rangle |
\ll_\vep
y^{1/2-\vep}
\cdot  \|\mathcal S_1(v_1)\|\cdot \| \mathcal S_1 (v_2)\|
,
\qquad\text{ as }\qquad y\to0. 
$$
\end{Prop}

\subsection{Approximations to Integrals 
over closed horocycles}
\

As $\gd>1/2$, there
exists a unique positive $L^2$-eigenfunction $\phi_0$ of the Laplace operator
$\Delta=-y^2\left(\partial_{xx}+\partial_{yy}\right)$ on $\G\ba \bH$ with smallest eigenvalue $\delta (1-\delta)$ and of unit norm 
 \cite{Patterson1976}.
The spectrum of $\gD$ acting on $L^2(\G\bk\bH)$ has finitely many discrete points below $1/4$ and is purely continuous above $1/4$  \cite{LaxPhillips1982}.

Order the other discrete eigenvalues $\gl_j=s_j(1-s_j)$, with $s_j>1/2$, $j=1,\dots,k$ and let $\phi_j$ denote the corresponding eigenfunction with $\|\phi_j\|=1$.
 For uniformity of notation, set $s_0:=\gd$.

The map $g\mapsto g(i)$ gives an identification of $G/K$ with $\bH$. Below
 a $K$-invariant function $\phi$ of $G$ may be considered
as a function on the upper half plane $\bH$ by
 $\phi(x+iy):=\phi(n_xa_y)$ and vice versa.

Let $dg$ denote the Haar measure of $G$ given by
$$d(n_xa_yk)=y^{-2}dxdydk$$
where $dk$ is the probability Haar measure on $K$, and $dx$ and $dy$ are Lebesque measures.
Hence for $\psi_1, \psi_2\in L^2(\G\ba G)^K$,
$$
\<\psi_{1},\psi_{2}\>
=
\int_{\G\ba G} \psi_1(g) \overline{\psi_2(g)}\;  dg
=
\int_{\G\ba \bH}\psi_1(x+iy) \overline{\psi_2(x+iy)}\;  dx\frac{dy}{y^2} .
$$

\begin{Prop}\label{matco}

For any smooth $\psi_1\in L^2(\G\bk G)^{K}$ and $\psi_2\in  C_c^\infty(\G\bk G)$,
\begin{align*}
\langle a_y \psi_1, \psi_2\rangle
&=
\sum_{j=0}^k
\langle \psi_1, \phi_j\rangle
 \langle a_y \phi_j, \psi_2\rangle +
%
   O_\vep( y^{1/2-\vep}\mathcal S_1(\psi_1)
\cdot \mathcal
 S_1(\psi_2) ).
\end{align*}
as $y\to0$.
\end{Prop}
\begin{proof}
Write
$$
L^2(\G\bk G)=W_{\gl_0} \oplus\cdots\oplus W_{\gl_k}\oplus V
$$
where $W_{\gl_j}$ is isomorphic  as a $G$-representation to the complementary series representation $V_{s_j}$, $\gl_j=s_j(1-s_j)$,
and  $V$
is tempered.
Write
$$
\psi_1
=
\langle \psi_1, \phi_0\rangle \phi_0+
\dots
+\langle \psi_1, \phi_k\rangle \phi_k+
\psi_1^\perp
.
$$
Since the $\phi_j$'s are the unique $K$-invariant vectors in $W_{\gl_j}$ (up to
scalar),
 we have $\psi_1^\perp\in V^K $. Hence
by Proposition \ref{decay}, for any $\vep>0$
and
$y\ll 1$,
\begin{align*} \langle a_y \psi_1, \psi_2\rangle
&=
\sum_{j=0}^k
\langle \psi_1, \phi_j\rangle
 \langle a_y \phi_j, \psi_2\rangle
+
 \langle a_y \psi_1^\perp , \psi_2 \rangle
\\
&=
\sum_{j=0}^k
\langle \psi_1, \phi_j\rangle
 \langle a_y \phi_j, \psi_2\rangle
+
O_\vep(y^{1/2-\vep} \mathcal S_1 (\psi_1) \cdot \mathcal S_1 (\psi_2))
\end{align*}
since $\mathcal S_1(\psi_1^\perp) \ll \mathcal S_1(\psi_1)$.
 %
%
\end{proof}


%

The main goal of this  subsection
is to study  the averages of the $\phi_j$'s along the translates $(N\cap\G)\ba N a_y$.

We first need to recall some geometric facts.
The limit set $\Lambda(\G)$ of $\G$ is the set of all accumulation points of
an orbit $\G(z_0)$ for some $z_0\in \bH$. As $\G$ is discrete,
it easily follows that $\Lambda(\G)$ is a subset of $\partial_\infty(\bH)=\br \cup\{\infty\}$.

 Geometrically, $N(i)\subset \bH$ is a horocycle based at $\infty$.
Since $\G$ is finitely generated,
 the closedness of its projection to $\G\ba \bH$ implies either
that $\infty$ is a parabolic fixed point, that is, $N\cap \G$ is non-trivial,
or that $\infty$ lies outside the limit set $\Lambda(\G)$ \cite{Dalbo2000}.

We define for each $0\le j\le k$
\be\label{PS}
\phi_j^N(a_y):=\int_{\G\cap N\ba N}\phi_j(na_y)dn .
\ee

 \begin{theorem}
\label{ppoo}
 For any $j=0,\dots,k$,
 the integrals in \eqref{PS} converge absolutely. Moreover
there exist  constants $c_{j}$ and $d_{j}$, depending on $\phi_j$, such that
$$
\phi_j^N(a_y)
=
c_j y^{1-s_j}+ d_j y^{s_j} .
$$
Furthermore, $c_0>0$.
\end{theorem}
\begin{proof}
We first establish the convergence of the integral in \eqref{PS}.
If $\infty$ is a parabolic fixed point, that is,
 if $N\cap \G$ is a lattice in $N$, then the domain of this integral is finite, and of course the eigenfunctions
  of the Laplace operator are bounded, so we are done.

On the other hand, if $\infty\notin \Lambda(\G)$,
 then Theorem \ref{eigHor} applies, that for $y$ fixed,
$$
\phi_{j}(n_{x}a_{y})\ll_{y} (1+x^{2})^{-s_{j}}
.
$$
Then the integral $\int_{-\infty}^{\infty}\phi_{j}(n_{x}a_{y})dx$ clearly converges, since $s_{j}>1/2$.

Since
$$
\Delta \phi_j =s_j(1-s_j) \phi_j,
$$
it follows that
$$-y^2 \frac{\partial^2}{\partial y^2} \phi_j^N  =s_j(1-s_j)\phi_j^N .$$
The two independent solutions to this equation are $y^{s_j}$ and $y^{1-s_j}$.

Lastly, we must demonstrate that $c_{0}>0$.
If $\infty\notin \Lambda(\G)$, this is done as in \cite{MyThesis} (see also \cite[Equation (4.11)]{KontorovichOh2008}), by proving explicitly that
$$
\phi_0^N(a_y)
=
c_0\ y^{1-\gd}
,
$$
i.e. $d_{0}=0$. As $\phi_0$ is a positive function, this implies $c_0>0$.

 If on the other hand $\infty$ is a parabolic fixed point for $\G$, then
 as the Dirichlet domain for $\G$ is a finite sided polygon
  with $\infty$ as a vertex \cite{Beardon1983}, it follows that for some $Y_0\gg 1$,
$(N\cap\G)\bk N\times[Y_{0},\infty)$ injects to
$\G\bk\bH$.
 Therefore, if we had $c_{0}=0$ and hence $\phi_0(n_xa_y)=d_0 y^{\delta_\G}$, then
\begin{multline*}
1=
\|\phi_{0}\|^{2}
\ge
\int_{Y_{0}}^{\infty}
\int\limits_{n_x\in ( N\cap \G)\bk N}
|\phi_{0}(n_{x}a_{y})|^{2}
d{x}\frac{dy}{y^2}\\
\ge \int_{(N\cap \G)\ba N}1 \;dn\cdot
\int_{Y_{0}}^{\infty}
d_{0}^2 y^{2\gd -2}
{dy}
=
\infty
,
\end{multline*}
since $\gd>1/2$.

This contradiction gives the desired result.
\end{proof}

The following Proposition shows that to compute  $\phi_j^N$, it suffices to integrate over a bounded set $J$
; the error term, if any, is small.

If $N\cap\G$ is a lattice in $N$, set $J:=(N\cap\G)\bk N$.
Otherwise as $\infty\notin\Lambda(\G)$, $\Lambda(\G)$ is a compact subset of $\br$, and we
let   $J\subset (N\cap \G)\ba N=\R$ be an open bounded interval which contains
$\Lambda (\G)$.
In either case $J$ is a bounded interval.

\begin{Prop}\label{ppo} We have
\be\label{phi0B}
\phi_j^N(a_y)=
\int_J \phi_j(n_xa_y)\; dx +  O  (y^{s_j }),
\ee
as $y\to0$.
\end{Prop}
\begin{proof}
If
$N\cap\G$ is a lattice in $N$%
, then $\phi_j^N$ and $\int_J\phi_j$ 
are identical; there is nothing to prove.
Otherwise, denoting
by  $J^c$  the complement of $J$, we have
$$
\int_{n_{x}\in J^c} \phi_j(n_{x}a_{y})dn_{x} \ll y^{s_j}.
$$
by Theorem \ref{eigHor}, which
 completes the proof.
\end{proof}

Next, we approximate $\phi_j^N$ by smoothing further in a transverse direction.
%
%
Denote by
  $U_\e$ the ball of radius $\e$ about $e$ in $G$. Let $J$ be as in Proposition \ref{ppo}.

Denoting by $N^-$ the lower triangular subgroup of $G$, we note that
$NAN^-$ forms an open dense neighborhood of $e$ in $G$.

\begin{Def} \label{def}
{\rm
\
\begin{itemize}
 \item We fix a non-negative function $\eta\in C_c^\infty(\G\cap N\ba N)$ with
 $\eta=1$ on $J$. 
\item
Let  $U_{c_0}$ denote the $c_0$-neighborhood of the identity $e$,
and fix $c_0 >0$ so that
the
multiplication map
$$\supp (\eta) \times (U_{c_0}\cap AN^-)
\to \supp (\eta) (U_{c_0}\cap AN^-)\subset  \G\ba G$$ is a bijection onto
its image.
\item  For each $\e<c_0$, let $ r_\e $ be a non-negative smooth function
in $AN^-$ whose support is contained in
$$ W_\e:=(U_\e \cap A)(U_{c_0}\cap N^-) $$ and $\int_{W_{\e}}  r_\e \; d\nu  =1$.

\item We define the following function $\rho_{\eta, \e}$ on $\G\ba G$ which is $0$ outside $\op{supp}(\eta)U_{c_0}$, and  for $g=n_xan^- \in \op{supp}(\eta)
(U_{c_0}\cap AN^-)$,
$$\rho_{\eta, \e}(g) := \eta (n_x) \otimes r_\e(an^- ) .$$
            \end{itemize}}
\end{Def}

\begin{Prop}\label{hugh} We have for all small $0<\e\ll \e_0$ and for all $0<y < 1$,
 $$| \phi_j^N(a_y)- \langle a_y\phi_j, \rho_{\eta, \e} \rangle_{L^2(\G\ba G)}|
 \ll  \e\cdot  y^{1-s_j} . $$
\end{Prop}
\begin{proof}
This follows in the same way as Proposition 6.4 in \cite{KontorovichOh2008}.
\end{proof}

The next corollary follows immediately from Proposition \ref{hugh} and Theorem \ref{ppoo}.

\begin{cor}
\label{corhugh}
 For $\e<1$, and $j=0,1,\dots,k$,
$$
 \langle a_y\phi_j, \rho_{\eta, \e} \rangle_{L^2(\G\ba G)}
 \ll  y^{1-s_j} ,
 $$
 as $y\to0$.
\end{cor}


\subsection{Proof of Theorem \ref{equi}}

The proof is almost identical to the proof of Theorem 6.1 in \cite{KontorovichOh2008}.  We sketch the main steps.

We use the following lemma, which is standard in
 Sobolev theory.

\begin{Lem}\label{approx} For $\psi
\in C^\infty_c(\G\ba G)^K$,
there exists $ \psi
^\sharp
\in C^\infty_c(\G\ba G)^K$
such that \begin{enumerate}

\item

for any $\e<\e_0$ and $h\in U_\e$,
$$|\psi
(g)-\psi
(gh)|\le \e \cdot  \psi
^\sharp
(g)\quad \text{ for all $g\in \Gamma\ba G$}. $$
\item $\mathcal S_{m}( \psi
^\sharp
) \ll \mathcal S_{m+1}(\psi
)$ for each $m\in \n$, where
the implied constant depends only on $\supp(\psi
)$.          \end{enumerate}
\end{Lem}


\begin{Def} {\rm For a given $\psi\in C^\infty (\G\ba G)^K$ and
$\eta\in C_c(\G\cap N\ba N)$, define the function
$\I_\eta(\psi)$ on $G$ by
$$\I_{\eta}(\psi)(a_y):=\int_{n\in (\G\cap N)\ba N} \psi(n a_y) \eta  (n)\; dn .$$}
\end{Def}



\begin{Prop}\label{mat}
 Let $\psi
\in C^\infty(\G\ba G)^K$.
Then for any $0<y<1$ and any small $\e >0$,
$$|\I_{\eta}(\psi
)(a_y) -\langle a_y \psi
, \rho_{\eta, \e}  \rangle | \ll
(\e+y) \cdot \I_\eta (\psi
^\sharp
) (a_{y})
,
$$
where $\psi
^\sharp
$ is given by Lemma \ref{approx}.
\end{Prop}
\begin{proof}
This is the same as Proposition 6.6 in \cite{KontorovichOh2008}.
\end{proof}


By Proposition \ref{ppo}, we have that
$$
\phi_j^N (a_y)=\I_\eta(\phi_j)(a_y)+O(y^{s_j})
.
$$
For simplicity, we set $\rho_\e=\rho_{\eta, \e}$ where $\rho_{\eta, \e}=\eta\otimes r_\e$
is defined
as in Def. \ref{def}. Noting that $r_\e$ is essentially
an $\e$-approximation only in the $A$-direction,
we obtain that $\mathcal S_1 (\rho_\e)=O(\e^{-3/2})$.

Set $p=3/2$.
Fix $\ell$, a parameter to be chosen later. 
%
%
Setting $\psi_0(g):=\psi (g)$,
we define for $1\le i\le \ell $,
inductively
$$
\psi_i (g) := \psi_{i-1}^\sharp(g)
$$
where $\psi_{i-1}^\sharp$ is given by Lemma \ref{approx}.

Applying
Proposition \ref{mat} to each $\psi_i$, we obtain for $0\le i\le \ell -1$
$$
\I_{\eta} ( \psi_i)(a_y)=
 \langle a_y   \psi_i , \rho_\e\rangle
+O((\e +y) \I_{\eta} ( \psi_{i+1})(a_y) )
$$
and
$$
\I_{\eta }( \psi_\ell)(a_y)=
 \langle a_y   \psi_\ell , \rho_\e\rangle
+O((\e +y) \mathcal S_1({\psi_\ell })) .
$$
Note that  by Corollary \ref{matco},
we have for each $1\le i\le \ell  $
\begin{align*}
 \langle a_y   \psi_i , \rho_\e\rangle
 &=
\sum_{j=0}^k
\la \psi_i, \phi_j\ra \la a_y \phi_j, \rho_\e\ra
+
 O(y^{1/2}
\mathcal S_1 (\rho_\e) )
\\
&\ll_\psi
y^{1-\gd}
.
\end{align*}

Combining the above with
 Proposition \ref{hugh} and
Corollary \ref{corhugh}, we get that for any $y<\e$,
\beann
\I_{\eta} (\psi)(a_y)
&=&
 \langle a_y\psi, \rho_\e\rangle
+O\left(
 \sum_{i=1}^{\ell -1} \langle a_y\psi_i, \rho_\e\rangle (\e+y) ^k\right) +  O(
\mathcal S_{\ell +1} ( \psi) (\e+y) ^\ell )
\\
&=&
\sum_{j=0}^k
\la \psi, \phi_j\ra \la a_y \phi_j, \rho_\e\ra
+
 O(y^{1/2}
\mathcal S_1 (\rho_\e) )
+O\left(
 \sum_{i=1}^{\ell -1} y^{1-\gd} \e^k\right) +  O(
\e ^\ell )
\\
&=&
\sum_{j=0}^k
\la \psi, \phi_j\ra
 \phi_j^N(a_y)
+
O\left(
  \e\cdot  y^{1-\gd}
  \right)
+
 O_{\vep}(y^{1/2-\vep}
\e^{-p} )
+
O(
\e ^\ell )
,
\eeann
where the implied constants depend on the Sobolev norms of $\psi$.
Balancing the first two error terms and recalling $p=3/2$, one arrives at the optimal choice
$$
\e = y^{(\gd-1/2)/(1+p)}
.
$$
With this choice of $\e$, the first two error terms are
$
\ll_{\vep}
y^{{4-3\gd\over 5}-\vep}
$.
Lastly, we choose $\ell$ to make the final error term of the same quality, namely
$
%
\ell
 =
(4-3\gd )/(2\gd-1).
$

Therefore
$$\int_{( N\cap \G)\ba N} \psi (n a_y)\; dn
=
\sum_{j=0}^k
\la \psi, \phi_j\ra
 \phi_j^N(a_y)
 +
 O_{\psi,\vep}(y^{{
4-3\gd\over 5 }-\vep }),
$$
for any $\vep>0$.

This completes the proof of Theorem \ref{equi}.

\section{Proofs of the Counting Theorems}\label{countSec}
With Theorem \ref{equi} at hand, we now count
 establish
 the counting theorems; these are
Theorems \ref{thm:count1} and \ref{thm:count2}.


\subsection{Proof of Theorem \ref{thm:count1}}
Recall that  $Q$ is
a ternary indefinite quadratic form, 
$\SO_{Q}(\R)$, and 
$\G$ 
a  finitely generated discrete subgroup with $\gd_{\G} >1/2$.
Fix a non-zero vector
 ${\bx}_{0}\in\R^{3}$,
lying on the cone $Q=0$ such that the orbit $\Or={\bx}_{0}\G$ is discrete.
%

Let
$$
\cN(T):=
\# \{{\bx}\in \Or: \| {\bx}\|<T\}
 $$
where $\|\cdot \|$ denotes a Euclidean norm on $\br^3$.

Using the spin cover $\iota:\SL_2\to \SO_Q$, we may assume without loss of generality
that $\G$ is a finitely generated subgroup of $G:=\SL_2(\br)$ with $\delta_\G>1/2$.
We use the notation $N, A, K, a_y, n_x$, etc from the introduction.
Let $dg$ denote the Haar measure given by
$$d(n_xa_yk)=y^{-2}dxdydk$$
where $dk$ is the probability Haar measure on $K$, and $dx$ and $dy$ are Lebesque measures.

As $Q({\bx}_0)=0$, the stabilizer of ${\bx}_0$ in $G$ is conjugate to
 $N$ and hence by replacing $\G$ with a conjugate if necessary, we may assume
 without loss of generality that $N$ is precisely the stabilizer subgroup of ${\bx}_0$ in $G$.
It follows that ${\bx}_0 a_y=y^{-1} {\bx}_0$.
 Let $B_T$ be a $K$-invariant ball of radius $T$  in $\R^3$, 
and let $\chi_{T}$ be the characteristic function of this ball. Note that $\chi_{T}$ is right $K$-invariant, that is, $\chi_{T}({\bx} gk)=\chi_{T}({\bx} g)$,  for any $g\in G$ and  $k\in K$.
Also, as $N$ is the stabilizer of ${\bx}_{0}$, we have $\chi_{T}({\bx}_{0} ng)=\chi_{T}({\bx}_{0} g)$, for any $n\in N$, that is, $\chi_{T}$ is left $N$-invariant.

We define the following counting function on $\G\ba G$:
$$F_T(g):=\sum_{\gamma\in N\cap \G\bk \Gamma}\chi_{T}({\bx}_0\gamma g) .$$
\begin{Lem}\label{ab}
For any $\Psi\in C_c(\Gamma\bk G)^K$,
$$\la  F_T, \Psi  \ra =\int_{y>T^{-1}\|{\bx}_0\|} \int_{n_x\in \G\cap N\ba N} \Psi(n_xa_y) y^{-2} dx dy .$$
\end{Lem}
\begin{proof}
We observe that by unfolding and using the $K$-invariance of $\Psi$ and $\chi_{B_T}$,
\begin{align*}
\la  F_T, \Psi  \ra &=\int_{\G\ba G}\sum_{\G\cap N\ba \G}\chi_{B_T}({\bx}_0\gamma g) \Psi(g) dg
\\&=\int_{\G\cap N\ba G}\chi_{B_T}({\bx}_0 g) \Psi(g) dg
\\ &=\int_{\|{\bx}_0 a_y\|<T}\int_{\G\cap N\ba N}\Psi(na_y) y^{-2} dndy
.
\end{align*}
As ${\bx}_0 a_y=y^{-1} {\bx}_0$, the claim follows.
\end{proof}

 As before, we order discrete eigenvalues $0\le \gl_0<\gl_1 \le \gl_2\cdots \le \gl_k<1/4$
 of $\Delta$ on $L^2(\G\ba \bH)$
 with $\gl_j=s_j(1-s_j)$, with $s_j>1/2$, $j=1,\dots,k$ and let $\phi_j\in L^2(\G\ba \bH)$ denote the corresponding eigenfunction with $\|\phi_j\|=1$.

Recall by Theorem \ref{ppoo},
there exist  constants $c_{j}$ and $d_{j}$, depending on $\phi_j$, such that
$$ \phi_j^N(a_y)
=
c_j y^{1-s_j}+ d_j y^{s_j} .
$$
Furthermore, $c_0>0$.

By inserting the asymptotic formula for $\int_{\G\cap N\ba N} \Psi(na_y) dn$
from Theorem \ref{equi}, we deduce:
\begin{Prop}\label{h1} For  any $\Psi\in C_c^\infty(\G\ba G)^K$ and $\e>0$,
\begin{align*}
\<
F_{T}
,
\Psi
\>_{\G\bk G}
&= \sum_{j=0}^{k} \< F_{T} ,\phi_{j}\>_{\G\bk G} \<\Psi, \phi_j\>_{\G\bk G}
+
O_{\vep}\bigg(T^{{\frac 12}+\frac3{5}(\gd-{\frac 12})+\vep}\bigg)
\\ &=
\sum_{j=0}^k \la  \Psi, \phi_j \ra_{\G\bk G}  \left(
\frac{c_jT^{s_j}}{s_j\|{\bx}_0\|^{s_j}} +\frac{d_jT^{1-s_j}}{(1-s_j) \|{\bx}_0\|^{1-s_j}}\right)
+ O_\e(T^{\frac{1}{2}+\frac{3}{5}(\delta-\frac 12)+\e})
\end{align*}
where the implied constant depends only on a Sobolev norm of $\Psi$ and $\e$.
\end{Prop}

For all small $\eta>0$, consider an $\eta$-neighborhood $U_\eta$ of $e$ in $G$, which is $K$-invariant,
such that for all $T\gg 1$,
$$B_TU_\eta \subset B_{(1+\eta)T}\quad\text{and}\quad B_{(1-\eta) T}\subset \cap_{u\in U_\eta}B_T u .$$
Let $\psi_\eta \in C^\infty(\bH)$ denote a non-negative function supported on $U_\eta$
with $\int_G \psi dg =1$. We lift $\psi_\eta$ to $\G\ba G$ by
$$\Psi_\eta(g):=\sum_{\gamma\in \G}\psi_\eta(\gamma g) .$$
Then
\be \label{ft}\la F_{(1-\eta)T}, \Psi_\eta\ra \le F_T(e)\le
\la F_{(1+\eta)T}, \Psi_\eta\ra .\ee

On the other hand,
recalling that $\{X_{1},X_{2},X_{3}\}$ is an orthonormal basis for
the Lie algebra $\fg=\frak sl(2,\R)$ of $G$, we have
$$\la \Psi_\eta, \phi_j\ra =\phi_j(e) + O(\eta \sup_{g\in U_\eta} \sup_{i=1,2,3} X_i\phi_j (g))$$
where the implied constant is absolute.

Therefore Proposition \ref{h1} yields for $C_0=\frac{\phi_0(e)c_0}{\delta \|{\bx}_0\|^\delta}>0$,
$$\la  F_{(1-\eta)T}, \Psi_\eta\ra = C_0
T^{\delta} +O_\e (T^{s_1}+\eta T^{\delta}+\eta^{-A}T^{\frac{1}{2}+\frac{3}{5}(\delta-\frac 12)+\e})$$
for some $A>0$ and any $\e>0$,
where we used that the Sobolev norms of $\Psi_\eta$ grow at most polynomially in $\eta$; the implied constant now depends only on $\e$.
 Therefore by equating the last two terms,
we can choose $\eta=T^{-r}$ for some $r>0$ and obtain from \eqref{ft} that
$$F_T(e)=C_0T^{\delta} +O (T^{\delta -\zeta})$$
for some $\zeta>0$. This proves
 Theorem \ref{thm:count1}.

 \subsection{{Proof of Theorem \ref{thm:count2}:}}
As in the discussion preceding 
 Theorem \ref{thm:count2}, we may assume that $\G$ is a finitely generated subgroup
of $\SL_2(\z)$. We may also assume without loss of generality that
$g_0=e$ by replacing $\G$ by $g_0\G g_0^{-1}$ 
and hence
the stabilizer subgroup of ${\bx}_0$ in $G$
is precisely the upper triangular subgroup  $N$.
Recall the weight $\xi_T$ defined in Definition \ref{xiDef}.



Let $q\ge1$ be squarefree, and let $\G_{1}(q)$ be any group satisfying
$$
\G(q)\subset\G_{1}(q)\subset\G
$$
and
\be\label{eq:stabNq}
N\cap\G_{1}(q) =N\cap \G,
\ee
where $\G(q)=\{\g\in\G:\g\equiv I(q)\}$ is the ``congruence'' subgroup of $\G$ of level $q$.

We define the following $K$-invariant functions on $\G_{1}(q)\bk G$:
$$
F_{T}^{q}(g):=\sum_{\g\in ( N\cap\G)\bk \G_{1}(q)}\chi_{T}({\bx}_{0}\g g)
,
$$
and for
any fixed $\g_{0}\in\G$,
$$
\Psi_{\g_{0}}^{q}(g):=\sum_{\g\in  \G_{1}(q)}\psi(\g_{0}^{-1}\g g).
$$

\begin{prop}
 We have
\be\label{xisum}
\sum_{{\bx}\in{\bx}_{0} \G_{1}(q)}
\xi_{T}({\bx} \g_{0})
=
\<
F_{T}^{q}
,
\Psi_{\g_{0}}^{q}
\>_{\G_{1}(q)\bk G}
.
\ee

\end{prop}

\begin{proof}
Starting with the left hand side of \eqref{xisum}, we insert Definition \ref{xiDef}, use that $N$ stabilizes ${\bx}_{0}$, and that $N\cap \G=N\cap \G_{1}(q)$:
\beann
\sum_{{\bx}\in{\bx}_{0} \G_{1}(q)}
\xi_{T}({\bx} \g_{0})
&=&
\sum_{{\bx}\in{\bx}_{0}\G_{1}(q)}
\int_{ G} \chi_{T}({\bx}\g_{0} g)\psi(g)dg
\\
&=&
\sum_{\g\in(N\cap \G)\bk \G_{1}(q)}
\int_{ G}
\chi_{T}({\bx}_{0} \g g)
\psi(\g_{0}^{-1}g)
dg
\\
&=&
\int_{ G}
F_{T}^{q}(g)
\psi(\g_{0}^{-1}g)
dg
\\
&=&
\sum_{\g\in \G_{1}(q)}
\int_{\G_{1}(q)\bk G}
F_{T}^{q}(g)
\psi(\g_{0}^{-1}\g g)
dg
\\
&=&
\int_{ G}
F_{T}^{q}(g)
\Psi_{\g_{0}}( g)
dg
,
\eeann
where we used the definition of $F_{T}^{q}$, ``refolded'' (the reverse of the ``unfolding trick''), and used the definition of $\Psi_{\g_{0}}^{q}$.
\end{proof}

Applying Proposition \ref{h1} to each $\G_1(q)$ and noting that
 Sobolev norms of $\Psi_{\g_{0}}^{q}$ are same as
that of $\psi$, i.e., independent of $\g_0$ and $q$, we obtain:

\begin{Prop}\label{FTPsiq}
Let
$$
0<\gd(1-\gd)=\gl_{0}^{(q)}<\gl_{1}^{(q)}\le\cdots\le \gl_{k_{q}}^{(q)}<1/4
$$
denote the point spectrum in $L^{2}(\G_{1}(q)\bk G)$, and let $\phi_{0}^{(q)},\dots,\phi_{k_{q}}^{(q)}$ be the corresponding eigenfunctions of unit norm. Then for any $\vep>0$,
\beann
\<
F_{T}^{q}
,
\Psi_{\g_{0}}^{q}
\>_{\G_{1}(q)\bk G}
&=&
\sum_{j=0}^{k_{q}}
\<
F_{T}^{q}
,
\phi_{j}^{(q)}
\>_{\G_{1}(q)\bk G}
\<
\phi_j^{(q)}
,
\Psi^{q}_{\g_{0}}
\>
_{\G_{1}(q)\bk G}
\\
&&
+\
O_{\vep}\bigg(T^{{\frac 12}+\frac3{5}(\gd-{\frac 12})+\vep}\bigg)
,
\eeann
as $T\to \infty$. The implied constant depends on $\vep$ and a Sobolev norm
 of $\psi$, but not on $q$ or $\g_{0}$.
\end{Prop}



\begin{lem}\label{qto1a}
For $q'|q$ and $\phi^{(q')}\in L^2(\G_1(q') \ba G)$ of norm one,
consider the normalized old form $\phi^{(q)}$ in $L^2(\G_1(q)\ba G) $ of level $q'$:
\be\label{eq:phi1Scale}
\phi^{(q)}
=
{1\over
\sqrt{[\G_1(q'):\G_{1}(q)]}
}
\phi^{(q')}
.
\ee

Then for any $\gamma_0\in \G$,
\be\label{eq:qto1b}
\dfrac{
\<\phi^{(q)}, \Psi^{q}_{\g_{0}}\>_{\G_{1}(q)\bk G}
}{
\<\phi^{(q')}, \Psi^{q'}_{\g_{0}}\>_{\G_{1}(q')\bk G}
}
=
\dfrac{
\<\phi^{(q)}, F^{q}_{T}\>_{\G_{1}(q)\bk G}
}{
\<\phi^{(q')}, F^{q'}_{T}\>_{\G_{1}(q')\bk G}
}
=
{1\over
\sqrt{[\G_1 (q'):\G_{1}(q)]}
}
.
\ee
\end{lem}

\begin{proof}
Consider the inner product
\beann
\<\phi^{(q)}, \Psi^{q}_{\g_{0}}\>_{\G_{1}(q)\bk G}
&=&
\int_{\G_{1}(q)\bk G}\phi^{(q)}(g) \Psi^{q}_{\g_{0}}(g)dg
\\
&=&
\int_{ G}\phi^{(q)}(g) \psi(\g_{0}^{-1}g)dg
\\
&=&
{1\over \sqrt{[\G_1(q'):\G_{1}(q)]}
}
\int_{ G}\phi^{(q')}(\g_{0}g) \psi(g)dg
\\
&=&
{1\over \sqrt{[\G_1(q'):\G_{1}(q)]}
}
\<\phi^{(q')}, \Psi^{q'}_{\gamma_0}\>_{\G\bk G}
,
\eeann
where we unfolded, used \eqref{eq:stabNq} and the $\G_1 (q')$-invariance of $\phi^{(q')}$, and refolded.

\end{proof}



\begin{lem}\label{lemFT}
Let $q=q'q''$ and assume that
$$
\phi_{j}^{(q)}
=
{1\over \sqrt{[\G_{1}(q'):\G_{1}(q)]}}
\phi_{j}^{(q')}
.
$$
Then
$$
\<
F_{T}^{q},\phi_{j}^{(q)}
\>
=
{1\over \sqrt{
[\G:\G_{1}(q
)]
}}
\bigg(
c_{j}^{(q')}\
T^{s_{j}}
+
d_{j}^{(q')}
\
T^{1-s_{j}}
\bigg)
,
$$
where $c_{j}^{(q')}$ and $d_{j}^{(q')}$ are independent of $q''$.
\end{lem}

\begin{proof}
The inner product
$
\<
F_{T}^{q}
,
\phi_{j}^{(q)}
\>
$
can be unfolded again, giving
\beann
\<
F_{T}^{q}
,
\phi_{j}^{(q)}
\>
&=&
 \int_{y>\|{\bx}_0 \|/T}
\int_{( N\cap \G)\ba N} \phi_j^{(q)} (n a_y)\; dn\,
{dy\over y^{2}}
\\
&=&
{1\over \sqrt{[\G_{1}(q'):\G_{1}(q)]}}
 \int_{y>\|{\bx}_0 \|/T}
\int_{( N\cap \G)\ba N} \phi_j^{(q')} (n a_y)\; dn\,
{dy\over y^{2}}
\\
&=&
{1\over \sqrt{[\G_{1}(q'):\G_{1}(q)]}}
 \int_{y>\|{\bx}_0 \|/T}
\bigg(
c_{j}^{(q')}y^{1-s_{j}}
+
d_{j}^{(q')}y^{s_{j}}
\bigg)
{dy\over y^{2}}
\\ &=&
{\sqrt{[\G:\G_{1}(q')]}\over \sqrt{[\G:\G_{1}(q)]}}
 \int_{y>\|{\bx}_0 \|/T}
\bigg(
c_{j}^{(q')}y^{1-s_{j}}
+
d_{j}^{(q')}y^{s_{j}}
\bigg)
{dy\over y^{2}}
\eeann
where we used Theorem \ref{ppoo} as well as the identity
 $$[\G:\G_{1}(q)]=[\G:\G_{1}(q')][\G_{1}(q'):\G_{1}(q)] .$$ The claim follows from a simple computation and renaming the constants.
\end{proof}



Recall from Definition \ref{fBdef} that
 square-free $q$ are to be decomposed as $q=q'q''$ with $q'\mid \fB$ and $(q'',\fB)=1$. 
%
%
Let
$$
0<\gd(1-\gd)=\gl_0^{(q)}<\gl_1^{(q)}\le\cdots\le\gl^{(q)}_{k_{q}}<1/4
$$
be
the eigenvalues of the Laplacian acting on $L^2(\G(q)\bk \bH)$. The eigenvalues below $\gt(1-\gt)$ are  all oldforms coming from level $1$, with the possible exception of finitely many eigenvalues coming from level $q'\mid\fB$.

For ease of exposition, assume the spectrum below $\gt(1-\gt)$ consists of only the base eigenvalue $\gl_{0}=\gd(1-\gd)$ corresponding to $\phi^{(q)}$, and one newform $\tilde\phi^{(q)}$ from the ``bad'' level $q'\mid\fB$. The general case is a finite sum of such terms.

Combining \eqref{xisum} and Proposition \ref{FTPsiq} with Lemmata \ref{qto1a} and \ref{lemFT} gives
\beann
\sum_{{\bx}\in{\bx}_{0}\G_{1}(q)} \xi_{T}({\bx} \g_{0})
&=&
{1\over [\G:\G_{1}(q)]}
\<F_{T},\phi^{(1)}\>_{\G\bk G}
\<\phi^{(1)},\Psi^{1}\>_{\G\bk G}
\\
&&
+
{[\G:\G_{1}(q')]\over [\G:\G_{1}(q)]}
\<F_{T},\tilde\phi^{(q')}\> _{\G_{1}(q')\bk G}
\<\tilde \phi^{(q')},\Psi^{q'}_{\g_{0}}\>_{\G_{1}(q')\bk G}
\\
&&
+
O_{\vep}(T^{\gt+\vep}+T^{{\frac 12}+\frac35(\gd-{\frac 12})+\vep})
.
\eeann
Setting
$$
\cE(T,q',\g_{0}):=
{[\G:\G_{1}(q')]}
\<F_{T},\tilde\phi^{(q')}\> _{\G_{1}(q')\bk G}
\<\tilde \phi^{(q')},\Psi^{q'}_{\g_{0}}\>_{\G_{1}(q')\bk G}
,
$$
the proposition follows by recognizing the main term as the main contribution to $\Xi(T)$.

This completes the proof of Theorem \ref{thm:count2}.

\section{Proofs of the Sieving Theorems}\label{sieving}

We now consider
$$Q=x^2+y^2-z^2$$ and fix
a finitely generated subgroup $\G <\SO_Q(\z)$ with $\delta_\G>1/2$.
Let ${\bx}_0\in \z^3\setminus\{0\}$ with $Q({\bx}_0)=0$.

Again by considering the spin cover $ G:=\SL_2 \to \SO_Q$ over $\q$,
we may assume without loss of generality that $\G$ is a finitely generated subgroup
of $\SL_2(\z)$.

Let $F$ be a polynomial which is integral on the orbit $\Or={\bx}_0\G$.


\subsection{Strong Approximation}
\

We first pass to a finite index subgroup of our original $\G$
which is chosen so that its projection to $\SL_2(\Z/p\Z)$ is either the identity or all of $\SL_2(\Z/p\Z)$.
\begin{lemma}
There exists an integer $N\ge1$ so that
$$\G(N)=\{\g\in\G:\g\equiv I(\mod N)\}$$
 projects onto $\SL_2(\Z/p\Z)$ for $p\nmid N$. Obviously the projection of $\G(N)$ in $\SL_2(\Z/p\Z)$ for $p|N$ is the identity.
\end{lemma}

This follows from Strong Approximation; see, e.g. \cite[\S2]{Gamburd2002}. 
From now on we replace $\G$ by $\G(N)$. We use Goursat's Lemma (e.g. \cite{Lang2002}, p. 75) to have a similar statement for the reduction modulo a square-free parameter:
\begin{theorem}[Thm 2.1 of \cite{BourgainGamburdSarnak2008}]\label{strongApp}
There exists a number $\fB$ such that if
 $q=q'q''$ is square-free
 with $q'|\fB$ and $(q'',\fB)=1$ then the projection of $\G$ in $\SL_2(\Z/q\Z)$ is
isomorphic to
 $
 \SL_2(\Z/q''\Z)$.
\end{theorem}

\subsection{Executing the sieve}
Choose a weight $\xi_T$ as in Definition \ref{xiDef}.
Let $\cA=\cA(T)=\{a_n(T)\}$ where
$$
a_n(T):= \sum_{{\bx}\in\Or\atop F({\bx})=n}\xi_T({\bx}).
$$
Then
$$
|\cA(T)| = \sum_n a_n(T) = \sum_{{\bx}\in\Or}\xi_T({\bx})
=
\Xi(T)
=:\cX,
$$
where the leading term  $\cX\asymp T^{\gd}$. 

For $q$ square-free
$$
|\cA_q(T)| = \sum_{n\equiv0(q)} a_n(T) =  \sum_{{\bx}\in\Or\atop F({\bx})\equiv0(n)}\xi_T({\bx})
.
$$
Let $\G_{{\bx}_{0}}(q)$ be the subgroup of $\G$ which stabilizes ${\bx}_{0}\mod q$, i.e.
$$
\G_{{\bx}_{0}}(q):=\{ \g\in\G:{\bx}_{0}\g\equiv\g(q) \}
.
$$
Then
$$
|\cA_q(T)|
 =
  \sum_{\g'\in\G_{{\bx}_{0}}(q)\bk\G\atop F({\bx}_{0}\g)\equiv0(q)} \left( \sum_{\g\in \op{Stab}_\G({\bx}_0)\bk \G_{{\bx}_{0}}(q)}\xi_T({\bx}_0\g\g')\right).
$$

Let
$$
|\Or^{F}(q)|:= \sum_{\g\in\G_{{\bx}_{0}}(q)\bk\G\atop F({\bx}_{0}\g)\equiv0(q)}1
,
$$
so that using Theorem \ref{thm:count2} gives
$$
|\cA_q(T)|
=
 |\Or^{F}(q)|
 \left(
 {1\over [\G:\G_{{\bx}_{0}}(q)]} (\cX + \cX_{q'})+ O_{\vep}(T^{\vep} |\Or^F(q)| (T^{\gt}+T^{{\frac 12}+\frac35(\gd-{\frac 12})})
\right)
,
$$
with
\be\label{badbnd}
\cX_{q'}\ll \cX^{1-\gz},
\ee
for some $\gz>0$, uniformly in $q'$.
Define
$$
g^F(q) := {|\Or^F(q)|\over [\G:\G_{{\bx}_0}(q)]}.
$$
\begin{Lem}\label{gis}
Assume the pair $(\Or,F)$ is strongly primitive.
Then for $q$ square-free, $g^F(q)$ is completely multiplicative, and $g^F(p)<1$. Furthermore,
\begin{enumerate}
\item
For $F_{\cH}({\bx})=z$, we have
\be
g^{F_\cH}(p) =
\begin{cases}{2p^{-1} + O(p^{-3/2}),}&\text{if $p\equiv1(4)$,}\\
{0,}&\text{otherwise.}
\end{cases}\ee
\item
For $F=F_\cA=\frac1{12}xy$, we have $g^{F_\cA}(p) =4p^{-1} + O(p^{-3/2})$.
\item
For $F=F_\cC=\frac1{60}xyz$, we have
\be
g^{F_\cC}(p) =
\begin{cases}
{4p^{-1} + O(p^{-3/2}),}&\text{if $p\equiv1(4)$,}\\
{6p^{-1} + O(p^{-3/2}),}&\text{if $p\equiv3(4)$.}
\end{cases}\ee
\end{enumerate}
\end{Lem}

\begin{proof}
For unramified $q=q_1q_2$ with $q_1$ and $q_2$ relatively prime and both prime to $\fB$, then  $\Or(q)$, the orbit of ${{\bx}_0}$ mod $q$ is equal to $\Or(q_1)\times\Or(q_2)$ in $(\Z/q_1\Z)^3\times(\Z/q_2\Z)^3=(\Z/q\Z)^3$.

For ramified $q=q'q''$, with $q''|\fB$, then $\G$ projects onto the identity mod $q''$, so $\Or(q'')$ is just one point, i.e. $\Or(q'')=\{{{\bx}_0}\}$. It also follows in this case that $\Or^F(q'q'')$ is isomorphic to $\Or^F(q')\times \Or^F(q'')$.

Since $[\G:\G_{{\bx}_0}(q)]=|\Or(q)|$, we have shown that
$$
g^F(q)={|\Or^F(q)|\over |\Or(q)|}
$$
is multiplicative, and thus
is determined by its values on the primes (only square-free $q$ are ever used).

From the assumption that the pair $(\Or,F)$ is strongly primitive, it immediately follows that $|\Or^F(p)|< |\Or(p)|$, i.e. $g^F(p)<1$. Notice that if $p|\fB$ then $|\Or(p)|=1$ and $|\Or^F(p)|=0$, so $g^F(p)=0$.

It remains to compute the values of $g^F$ on primes $p\mid\fB$.
Denote by $V$ the cone defined by $Q=x^2+y^2-z^2=0$, minus the origin. I.e.
$$
V=\{ (x,y,z)\neq(0,0,0): x^2+y^2-z^2=0\}.
$$
For $F=F_\cH=z,$ let
$$
W_1
=
\{ {\bx}\in V: F_\cH({\bx})=0\}
=
\{ (x,y,0)\neq(0,0,0): x^2+y^2=0\}
.
$$
As $V$ is a homogeneous space of $G$ with a connected stabilizer, we have
$$
\Or(p)=V(\F_p),\quad \text{and hence } \Or^{F_\cH}(p)=W_1(\F_p).
$$
We can easily calculate $|V(\F_p)|=p^2+O(p^{3/2})$. If $p\equiv3(4)$, then $W_1(\F_p)$ is empty. If $p\equiv1(4)$, then $W_1$ is the disjoint union of the two lines $\{ (x,y)\neq0: x=\pm\sqrt{-1}y\}$, each of cardinality $p-1$. This proves claim (1).

For $F=F_\cA=\frac1{12}xy$, we set
\beann
W_2
&:=&
\{ {\bx}\in V: F_\cA({\bx})=0\}
\\
&=&
\{ (0,y,z)\neq(0,0,0): y^2-z^2=0\} \sqcup
\{ (x,0,z)\neq(0,0,0): x^2-z^2=0\}
\eeann
Thus $W_2$ is the disjoint union of four  lines, $x=\pm z$, $y=\pm z$, proving claim (2).

For $F=F_\cC=\frac1{60}xyz$, we see immediately that
$$
W_3:=
\{ {\bx}\in V: F_\cC({\bx})=0\} = W_1\sqcup W_2,
$$
proving claim (3).
\end{proof}

From Lemma \ref{gis}, the computation leading to \eqref{locDens} is a classical exercise (see e.g. \cite{Landau1953}), with sieve dimensions
\be\label{gkis}
\gk=1, 4\text{ and $5$ for $F=F_\cH$, $F_\cA$ and $F_\cC$, respectively.}
\ee

Define $r_q:= |\Or^F(q)|T^{\gt+\vep}$. From the proof of Lemma \ref{gis}, we have $|\Or^F(p)|\ll p$, so
$$
\sum_{q<\cX^\tau \atop q \text{ squarefree}}
4^{\nu(q)}|r_q| \ll_\vep \cX^{2\tau+\vep} T^{\gt}.
$$
As $\cX\sim c T^{\gd}$, this error term is admissible, that is, satisfies \eqref{taudef}, for any
\be\label{tauis}
\tau < {\gd-\gt\over 2\gd}.
\ee
The elements $a_n(T)$ are zero for $n\gg T\gg\cX^{1/\gd}$, so \eqref{mudef} is satisfied for
\be\label{muis}
\mu > {2\over \gd-\gt}
.
\ee
We are not yet ready to apply Theorem \ref{sieve} since our sequence $\cA$ satisfies
$$
|\cA_q| = g(q) \cX + g(q) \cX_{q'} + r_q,
$$
the middle term of which is a nuisance. We define a new sequence $\cA'$ via
$$
S(\cA'\cW)=\sum_q \gw(q) \left( g(q) \cX  + r_q\right) = \sum_n a'_n\gw(n)
,
$$
and notice that
\bea\label{Sdiff}
| S(\cA\cW)-S(\cA'\cW)|
&=&
\left| \sum_q \gw(q)g(q) \cX_{q'}\right|
\ll \cX^{1-\eta}
\eea
by virtue of \eqref{badbnd}.

Now we may apply Theorem \ref{sieve} to $S(\cA'\cW)$.
 For any $R$ satisfying \eqref{Ris}, have
$$
S(\cA'\cW)\gg{\cX\over \log^{\gk} \cX},
$$
where $\gk$ is determined in \eqref{gkis}, according to the choice of $F\in\{F_{\cH},F_{\cA},F_{\cC}\}$.
Together with \eqref{Sdiff}, we now have that
$$
S(\cA\cW)\gg {\cX\over \log^\gk \cX},
$$
as desired. Having verified the sieve axioms, the upper bound of the same order of magnitude follows from a standard application of a combinatorial sieve, see e.g. \cite[Theorem 2.5]{Kontorovich2009} where the details are carried out. 

\subsection{Explicit values of $R$}\label{sec:valsR}
It remains to determine values of $R$ for which the above discussion holds.

The values of $\ga_\gk$ and $\gb_\gk$ in Theorem \ref{sieve} can be tabulated, see for instance Appendix III on p. 345 of \cite{DiamondHalberstamRichert1988}. The sets $\cA$ appearing in this paper  have  sieve dimensions $\gk=1, 4$ and $5$; for these values, we have
\be\label{abvals}
\ga_1=\gb_1=2, \ga_4=11.5317.., \gb_4=9.0722.., \ga_5=14.7735.., \gb_5=11.5347...
\ee
We will also need precise estimates on the functions which appear in \eqref{Ris}.
Although these are difficult to extract by hand, the following procedure is quite effective in practice. Usually, $u$ is chosen so that $\tau u$ is near $1$, and $v$ so that $\tau v$ exceeds $\ga_\gk$. Precisely, for any $\gz\in(0,\gb_\gk)$, set
$$
\tau u=1+\gz-\gz/\gb_\gk, \quad \tau v=\gb_\gk/\gz + \gb_\gk -1.
$$
Then by Halberstam-Richert \cite{HalberstamRichert1974}, equations (10.1.10), (10.2.4) and (10.2.7), we obtain
$$
{\gk \over f_\gk(\tau v)} \int_1^{v/u}F_\gk(\tau v-s)\left(1-\frac uv s\right){ds\over s}
\le (\gk+\gz)\log\frac{\gb_\gk}\gz - \gk + \gz\frac\gk{\gb_\gk}.
$$
Thus Theorem \ref{sieve} holds with
\be\label{Riz}
R> \mu (1+\gz-\gz/\gb_\gk) - 1 +
(\gk+\gz)\log\frac{\gb_\gk}\gz - \gk + \gz\frac\gk{\gb_\gk}=:m(\gz),
\ee
for any $0<\gz<\gb_\gk$. After inputting the values of $\mu$, $\gk$, $\ga_\gk$ and $\gb_\gk$, the minimum of $m(\gz)$ is easily determined by hand or with computer assistance.


\begin{table}
\begin{tabular}{|c|ccccc|c|}
	\hline
$F$ &  $(\G\cap N)\bk N$& $\gd$  & $\gt$  & $\mu$ &  $m(\gz)$ & $R$  \\
	\hline
$F_\cH$ & Any & $1$ & $5/6$ & $12$ & $13.93.. $ & $14 $\\
$F_\cH$ & Any & $\boxed{0.9992}$ & $5/6$ & $12.05$ & $13.99.. $ & $\boxed{14} $\\
$F_\cH$ & Any & $1$ & $39/64$ & $5.12$ & $6.48.. $ & $7 $\\
$F_\cH$ & Finite & $1$ & $1/2$ & $4$ & $5.22.. $ & $6 $\\
$F_\cH$ & Finite & $0.9265$ & $1/2$ & $4.69$ & $5.99.. $ & $6 $\\
$F_\cH$ & Infinite & $1$ & $1/2$ & $10$ & $11.8.. $ & $12 $\\
$F_\cH$ & Infinite & $0.991$ & $1/2$ & $10.2$ & $11.9.. $ & $12 $\\
	\hline
$F_\cA$ & Any & $1$ & $5/6$ & $12$ & $24.9.. $ & $25 $\\
$F_\cA$ & Any & $\boxed{0.99995}$ & $5/6$ & $12.0$ & $24.9.. $ & $\boxed{25} $\\
$F_\cA$ & Any & $1$ & $39/64$ & $5.12$ & $15.6.. $ & $16 $\\
$F_\cA$ & Finite & $1$ & $1/2$ & $4$ & $13.8.. $ & $14 $\\
$F_\cA$ & Finite & $0.98805$ & $1/2$ & $4.1$ & $13.9.. $ & $14 $\\
$F_\cA$ & Infinite & $1$ & $1/2$ & $10$ & $22.4.. $ & $23 $\\
$F_\cA$ & Infinite & $0.97895$ & $1/2$ & $10.4$ & $22.9.. $ & $23 $\\
	\hline
$F_\cC$ & Any & $1$ & $5/6$ & $12$ & $28.7.. $ & $29 $\\
$F_\cC$ & Any & $\boxed{0.99677 }$ & $5/6$ & $12.2$ & $28.99.. $ & $\boxed{29} $\\
$F_\cC$ & Any & $1$ & $39/64$ & $5.12$ & $18.7.. $ & $19 $\\
$F_\cC$ & Finite & $1$ & $1/2$ & $4$ & $16.7.. $ & $17$\\
$F_\cC$ & Finite & $0.981675$ & $1/2$ & $4.2$ & $16.99.. $ & $17 $\\
$F_\cC$ & Infinite & $1$ & $1/2$ & $10$ & $25.9.. $ & $26 $\\
$F_\cC$ & Infinite & $0.99905$ & $1/2$ & $10.02$ & $25.9.. $ & $26 $\\
	\hline
\end{tabular}
\vskip.1in \caption{Values of $R$ depending on $\gd$, $\gt$, and whether $N\cap \G$ is assumed to be a lattice in $N$.}
\label{FigTab}
\end{table}

For the function $F_{\cH}=z$, we have $\gk=1$, and $\ga_1=2=\gb_1$. The best value of $R$ is obtained for $\gd\to1$, where we can take $\gt=5/6$. Then \eqref{muis} gives $\mu>12$, and we have collected everything required to compute the minimum of $m(\gz)$ defined in \eqref{Riz}. We find that the minimum value is attained at $\gz=0.1203..$ with $m(\gz)=13.931..$.  Thus $R=14$ is the limit of our method. For $\gd>1-1/1250$ and $\gt=5/6$, we find the minimum value $m(0.1198..)=13.992$, which still allows $R=14$. For comparison, consider instead a finite co-volume congruence subgroup; then $\gd=1$ and can take $\gt=39/64$ by \cite{KimSarnak2003}. This gives $m(0.238..)=6.48..$, allowing $R=7$. If one could take $\gt$ arbitrarily close to $1/2$, the above calculation gives $m(0.292..)=5.216..,$ or $R=6$.\\

Table \ref{FigTab} summarizes the discussion above and extends it to the other choices $F_{\cA}=\frac1{12}xy$  and $F_{\cC}=\frac1{60}xyz$, with various possibilities for $\gd$, $\gt$, and whether $\G\cap N$ is a lattice in $N$. For comparison, we also show the spectral gap $\gt=39/64$ \cite{KimSarnak2003} for congruence subgroups of $\SL(2,\Z)$.  These values of $R$ are precisely those quoted in Theorems \ref{numbs}, \ref{numbs1}, and \ref{numbs2}, in particular proving Theorem \ref{thm:intro}.
\\

\

\

\appendix

\section{Proof of Theorem \ref{thm:hors}}\label{horoints}

Recall the notation \eqref{eq:NAdef}.
Let $\G<G$ be a discrete, finitely generated subgroup with critical exponent $\gd>1/2$.
Let $\phi\in L^{2}(\G\bk\bH)$ be an eigenfunction of the hyperbolic Laplace-Beltrami operator $\gD$ with eigenvalue $\gl=s(1-s)<1/4$ and $s>1/2$.
\\

The aim of this section is to reproduce the proof of Theorem \ref{thm:hors}, which was demonstrated in \cite{MyThesis}. We give the statement again. 
\\

\begin{theorem}[\cite{MyThesis}]\label{eigHor}
Assume that the volume of $\G\bk\bH$ is infinite, and that  the  horocycle $(N\cap\G)\bk N$ is closed and
infinite. 
There exist $x_{0}>0$ and $y_{0}<1$ such that if $|x|>x_{0}$ and $y<y_{0}$, then
$$
\phi
(n_{x}a_{y})
\ll
\left(
{y\over
x^{2}+y^{2}
}
\right)
^{s}
,
$$
as $|x|\to\infty$ and $y\to0$.
\\
\end{theorem}

The proof of this fact  
is reminiscent of
the arguments given in Patterson \cite{Patterson1975} and Lax-Phillips \cite{LaxPhillips1982} showing that a square-integrable eigenfunction of the Laplacian acting on an infinite volume surface must have eigenvalue $\gl<1/4$, i.e. the spectrum above $1/4$ is purely continuous. The key ingredient is that being $L^2$ forces an asymptotic formula for the rate of decay of the eigenfunction as it approaches the free boundary in the flare. 

\subsection{Fourier Expansion in the Flare}

 Recall that a ``flare'' in the fundamental domain is a region bounded by two geodesics
  , containing a free boundary. Concretely, after conjugation in $\SL_2(\R)$, we can assume that our group $\G$ contains the fixed hyperbolic cyclic subgroup generated by the element $\mattwo{\sqrt{k}}{0}{0}{1/\sqrt{k}}:z\mapsto kz$. So a flare domain is isometric to a domain of the form $\{z:1<|z|<k;0<\arg z<\ga\}$, where $\ga<\pi/2$. See Fig. \ref{FD}.

\begin{figure}
\includegraphics[width=2in]{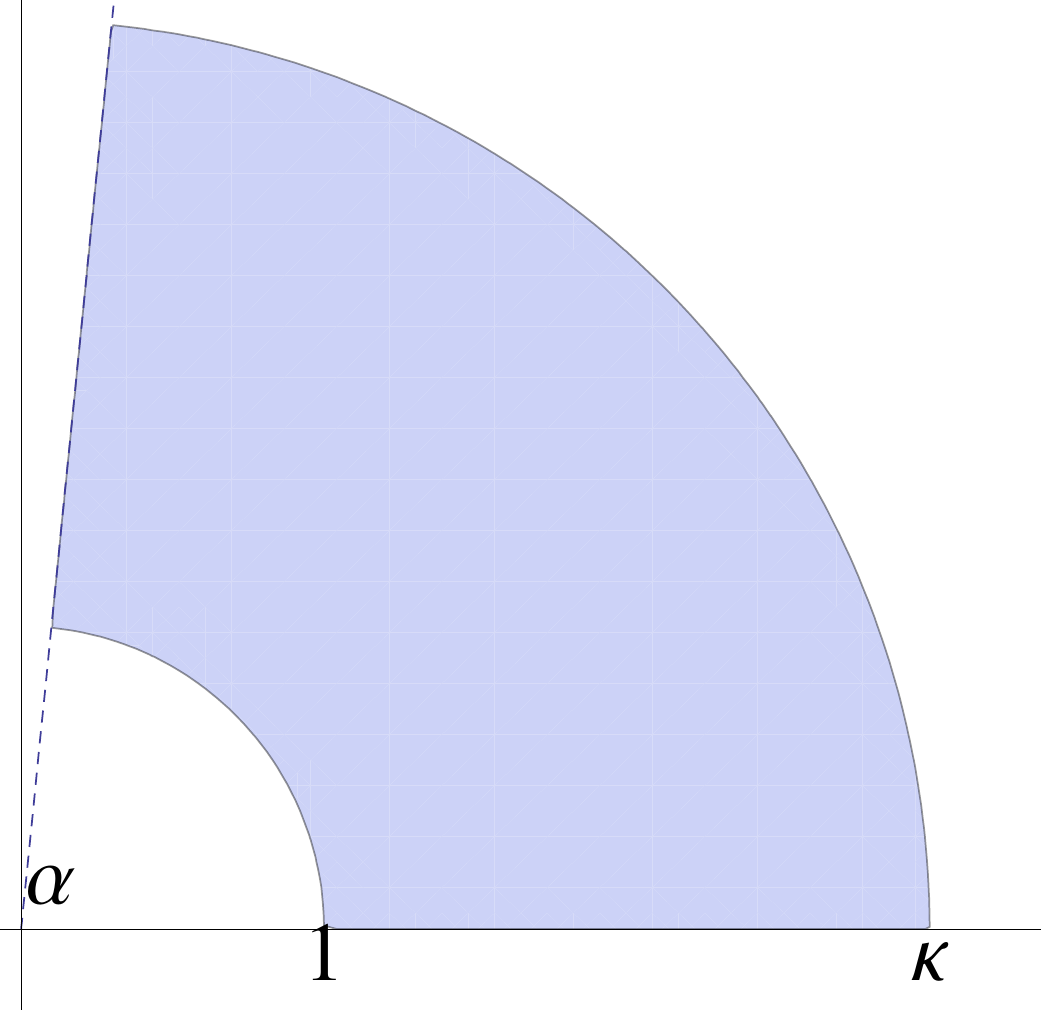}
\caption{A flare domain.}
\label{FD}
\end{figure}

Such a conjugation sends the point at infinity to some point $\xi\in[1,\gk]$, and a horocycle at infinity to a circle $h$ tangent to $\xi$.
See Fig. \ref{FCD}.

Let $\phi\in L^2(\GbkH)$, $||\phi||_2=1$, with $\gD\phi=s(1-s)\phi$. We proceed with the Fourier development of $\phi$ using polar coordinates in this domain.
As we are in the flare,  $\phi(\gk z)=\phi(z)$, $\gk>1$.
Write $z=re^{i\gt}$ in polar coordinates and separate variables:
\benn
\phi(z)=\phi(r,\gt)=f(r)g(\gt),
\eenn
with $f(\gk r)=f(r)$. Expand $f$ in a (logarithmic) Fourier series:
\benn
\phi(r,\gt)=\sum_{n\in\Z} g_n(\gt) e^{2\pi i n\log r/\log \gk}.
\eenn
Then the solution to the differential equation induced on $g_n$ is (see e.g.
\cite{Gamburd2002}, pages 180--181):
\benn
g_n(\gt)=c_n\sqrt{\sin \gt} P^{\mu}_{\nu}(\cos \gt),
\eenn
where $c_n\in\C$ are some
coefficients
and  $P^{\mu}_{\nu}$ is the associated Legendre function of the first kind with
\bea\label{munu}
\mu&=&{\frac 12}-s \\
\nonumber
\nu&=&-{\frac 12}+{2\pi i n\over \log \gk}.
\eea

We have proved
\begin{prop}
There are some coefficients $c_{n}\in\C$ such that
\be\label{f}
\phi(r,\gt)=\sum_n c_n \ e^{2\pi i n\log r/\log \gk} \sqrt{\sin \gt} P^{\mu}_{\nu}(\cos \gt)
\ee
in the flare, with $1\le r\le\gk$ and $0\le \gt \le \ga<\pi/2$.
\end{prop}

Next we need bounds on the coefficients $c_n$.

\subsection{Bounds on the Fourier Coefficients}

\begin{figure}
\includegraphics[width=2in]{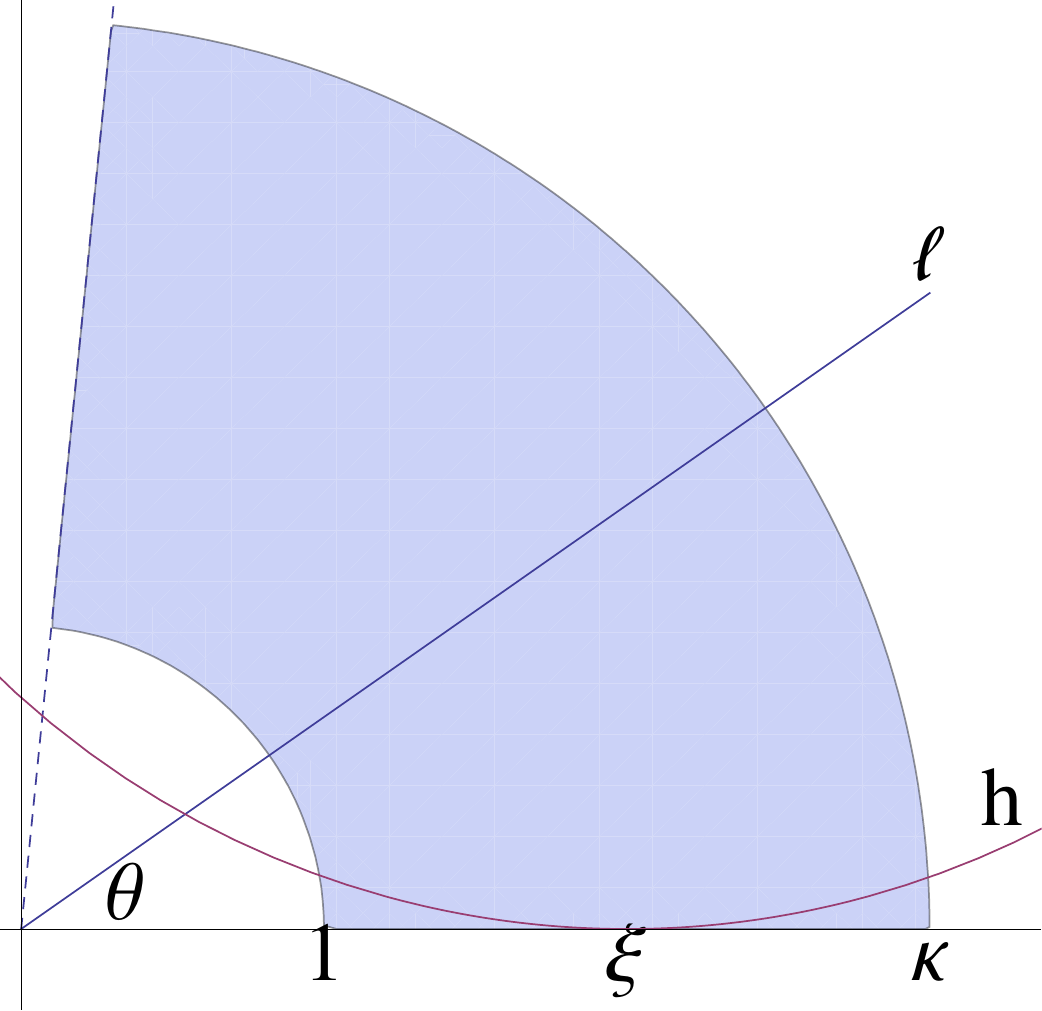}
\caption{The horocycle $h$ is based at the point $\xi$ which is in the free boundary. The line $\ell$ intersects the real line at angle $\gt$.}
\label{FCD}
\end{figure}

\begin{prop}\label{coeffBnd}
The coefficients $c_n$ in \eqref{f} satisfy
\benn
c_n\ll n^{s} e^{-\pi n\ga/\log\gk},
\eenn
as $n\to\infty$. The implied constant depends on $\ga$ and $\gk$.
\end{prop}
\begin{proof}
In Cartesian coordinates $z=x+iy$, the Haar measure is $dz={dx dy\over y^2}$. In polar coordinates
this becomes
\benn
dz =
  { dr d\gt\over r \sin^2\gt}.
\eenn

Input the expansion \eqref{f} into
 $||\phi||=1$,
 and consider only the contribution from the flare domain.
This gives
\bea\nonumber
1&\ge&\int_1^{\gk}\int_0^{\ga} |\phi(r,\gt)|^2 {dr d\gt\over r\sin^2 \gt}
\\
\nonumber
& =& \log\gk \sum_n |c_n|^2 \int_0^{\ga}|P^{\mu}_{\nu}(\cos\gt)|^2{d\gt\over\sin\gt}.
\\
\label{l2}
& \ge & \log\gk \sum_n |c_n|^2 \int_{\ga/2}^{\ga}|P^{\mu}_{\nu}(\cos\gt)|^2{d\gt\over\sin\gt},
\eea
where we have decreased the range of integration by positivity.

By Stirling's formula and the values of $\mu$ and $\nu$ in \eqref{munu},
\be\label{gammas}
{\G(\nu+\mu+1)\over\G(\nu+{3\over2})}
\gg n^{-s},
\ee
 for $n\gg1$. The implied constants depend only on $\gk$.

Next we record the elementary bound
\be\label{coses}
\cos\left[ (\nu+{1\over2})\gt + (\mu-{1\over2}) {\pi\over2}\right]
\gg e^{2\pi n\gt/\log\gk}.
\ee

Finally we have the formula (see \cite{GradshteynRyzhik2007} p.$1003$ $\# 3$),
\be\label{gr}
P^{\mu}_{\nu}(\cos\gt) ={2\over\sqrt{\pi}} {\G(\nu+\mu+1)\over\G\left(\nu+{3\over2}\right)} { \cos\left[ (\nu+{1\over2})\gt + (\mu-{1\over2}) {\pi\over2}\right] \over \sqrt{2\sin\gt}} \left(1+O\left({1\over\nu}\right)\right),
\ee
valid whenever
\begin{enumerate}
\item[(P1)] $\mu\in\R$, $|\nu|\gg 1$,
\item[(P2)] $|\nu|\gg|\mu|$,
\item[(P3)] $|\arg\nu|<\pi$,
\item[(P4)] $0<\vep<\gt<\pi-\vep$, and
\item[(P5)] $|\nu|\gg{1\over\vep}$.
\end{enumerate}
The big-Oh constant is absolute, depending on the implied constants above.

The
conditions (P1) and (P2) are immediately satisfied
 from the values of $\mu$ and $\nu$ in \eqref{munu}.
 The argument of $\nu$ approaches ${\pi\over2}$ for $n$ large, so (P3) is easily satisfied.  We will use this formula for $\gt$ in a fixed range away from zero, $\gt\in[\ga/2,\ga]$. 
 Thus (P4) is satisfied, and (P5) is equivalent to (P1).

Since $\gt$ is bounded away from zero, so is the factor $\sin\gt$ in the denominator of \eqref{gr}. Putting together \eqref{gammas}, \eqref{coses}  and \eqref{gr} gives
%
%
\be\label{p2}
|P^{\mu}_{\nu}(\cos\gt)| \gg n^{-s}  e^{2\pi n\gt/\log\gk} \gg n^{-s}  e^{\pi n\ga/\log\gk} .
\ee


Returning  to \eqref{l2},  consider the contribution from  just the $N$th coefficient
and use \eqref{p2}
\beann
1&
\ge
&
\log\gk \sum_n |c_n|^2 \int_{\ga/2}^{\ga}|P^{\mu}_{\nu}(\cos\gt)|^2{d\gt\over\sin\gt}\\
&\ge & |c_N|^2 \int_{\ga/2}^{\ga}|P^{1/2-s}_{-1/2+2\pi i N/\log\gk}(\cos\gt)|^2 d\gt\\
&\gg &|c_N|^2 N^{-2s} e^{2\pi N\ga/\log\gk},
\eeann
as $N\to\infty$.

%

This completes the proof of Proposition \ref{coeffBnd}.
\end{proof}

\subsection{Radial Bounds for the Eigenfunction}

Next we get bounds on the eigenfunction $\phi$ as the angle $\gt$ decreases to zero.

\begin{prop}\label{thbound}
Let $\phi$ be
as above
with eigenvalue $\gl=s(1-s)$, $s>1/2$. Then
$$
\phi(r,\gt)\ll \gt^s,
$$
as $\gt\to0$.
\end{prop}

\begin{proof}
For $\gt$ small,
$$
\sin\gt \asymp \gt ,
$$
and
$$
1-\cos\gt
\asymp
\gt^2.
$$

We require some more estimates. First we use the following standard bound on the Gauss hypergeometric series:
\benn
F(a,b,c;x)=1+O(\left|{abx\over c}\right|),
\eenn
valid for
\benn
|x|\max_{\ell\in\Z}\left|{(a+\ell)(b+\ell)\over(c+\ell)(1+|\ell|)}\right|\le{\frac 12}.
\eenn
In particular, with
\beann
a&=&-\nu={\frac 12}-{2\pi\i n\over\log\gk},\\
b &=&1+\nu=\bar{a},\\
c&=&1-\mu={\frac 12}+s, \text{ and}\\
x&=&{1-\cos\gt\over2}\ll\gt^2,
\eeann
the above gives:
\be\label{fbound}
F(-\nu,1+\nu,1-\mu;{1-\cos\gt\over2})\ll1,
\ee
whenever
\benn
n\ll{1\over\gt}.
\eenn

We require
 \cite{GradshteynRyzhik2007} p. 999 formula 8.702:
\beann
P_{\nu}^{\mu}(z) &=& {1\over\G(1-\mu)}\left(1+z \over 1-z\right)^{\mu/2} F(-\nu,\nu+1;1-\mu;{1-z\over2}).
\eeann
For $z=\cos\gt$, we have
\benn
\left(1+z \over 1-z\right)^{\mu/2}\asymp \gt^{-\mu}=\gt^{s-1/2},
\eenn
so together with \eqref{fbound} we arrive at
\be\label{pbound}
P^{\mu}_{\nu}(\cos\gt)\ll \gt^{-\mu}=\gt^{s-1/2}
\ee
for $n\ll{1\over\gt}$.

The analysis leading to \eqref{gr} also gives
\be\label{p}
|P^{\mu}_{\nu}(\cos\gt)| \ll  n^{-s}  e^{2\pi n\gt/\log\gk}  \gt^{-1/2}
\ee
in the range $n\gg{1\over\gt}$.

Thus we split the Fourier series as follows:
\beann
\phi(r,\gt)&=&\sum_n c_n e^{2\pi\i n\log r/\log \gk} \sqrt{\sin \gt} P^{\mu}_{\nu}(\cos \gt)\\
&=&\sum_{n\le X} + \sum_{n>X} \\
&=& S_1 + S_2,
\eeann
with $X\asymp {1\over\gt}$.

On $S_1$ we use the bound \eqref{pbound}:
\beann
|S_1|&\le&\sum_{n\le X} |c_n|  \sqrt{\sin \gt} |P^{\mu}_{\nu}(\cos \gt)|\\
&\ll& \gt^{1/2} \gt^{-\mu} \sum_n |c_n|\\
&\ll& \gt^{s},
\eeann
by the exponential decay of $c_n$ (clearly the series converges).
\\

On $S_2$, we use the bound \eqref{p}:\\
\beann
|S_2|&\le& \sum_{n>X} |c_n|  \sqrt{\sin \gt} |P^{\mu}_{\nu}(\cos \gt)|.\\
&\ll& \gt^{1/2} \sum_{n>X} n^s e^{-\pi n\ga/\log\gk} n^{-s}  e^{2\pi n\gt/\log\gk}  \gt^{-1/2}\\
&\ll&  \sum_{n>X}  e^{-\pi n(\ga-2\gt)/\log\gk} \\
&\ll&  \exp(-\pi X(\ga-2\gt)/\log\gk) \\
&\ll&  \exp(-{1\over \gt}\pi \ga/\log\gk),\\
\eeann
since $X\asymp{1\over\gt}$.

Combining
the exponential decay of $S_2$
with
the polynomial decay of $S_1$, we
arrive at
$$
\phi(r,\gt)\ll \gt^s,
$$
as $\gt\to0$.

This completes the proof of Proposition \ref{thbound}.
\end{proof}

\subsection{Proof of Theorem \ref{eigHor}}

Finally, we return to cartesian coordinates to convert the bound above into Theorem \ref{eigHor}. We require the following geometric analysis.
Recall Fig. \ref{FCD}, 
where $\xi$ is a point on the free boundary of the fundamental domain,
 $h$ is a horocycle tangent to $\xi$, and
 $\gt$ is the angle between the real line and the line $\ell$ from zero to infinity intersecting the horocycle, $h$ at a point $C$.

To return to cartesian coordinates
, we redraw our picture after the conformal mapping
\benn
z\to{z+\xi\over -z+\xi},
\eenn
which sends the triple $(0,\infty,\xi)\mapsto(1,-1,\infty)$. See Fig. \ref{MAD}.

\begin{figure}
\includegraphics{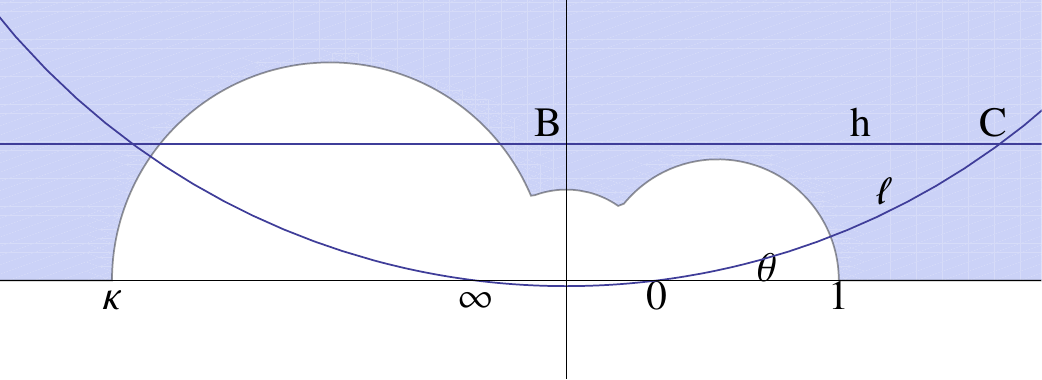}
\caption{The image of Fig. \ref{FCD} after a conformal transformation sending $\xi\mapsto\infty$, $0\mapsto1$ , and $\infty\mapsto-1$.
The point $C$ is a point of intersection of $h$ and $\ell$,
and $B$ is the intersection of $h$ with the $y$-axis.}
\label{MAD}
\end{figure}

Lines and circles are mapped to lines and circles, and angles of incidence are preserved. The horocycle tangent to $\xi$ is now the horizontal line, $h$ (tangent to $\xi$, which has been mapped to infinity). Similarly, the line $\ell$ from zero to infinity passing through the horocycle is now a circle passing through the same points, having the same angle of incidence, $\gt$, with the real line.

We reconstruct this configuration yet again in Fig. \ref{xfig}. Let $A$ be the center of the circle $\ell$, having radius $R=R(\gt)$, let $B$ be the intersection of the horocycle $h$ with the $y$-axis, $C$ the intersection of the horocycle with the circle, and let $D$ denote the origin.

It is easy to see through elementary geometry (since $\overline{A0}$ is tangent to the circle) that angle $0AD=\gt$. Looking at triangle $0AD$, we see that
\be\label{R}
R\sin\gt=\overline{0D}=1
.
\ee

\begin{figure}
\includegraphics[width=2in]{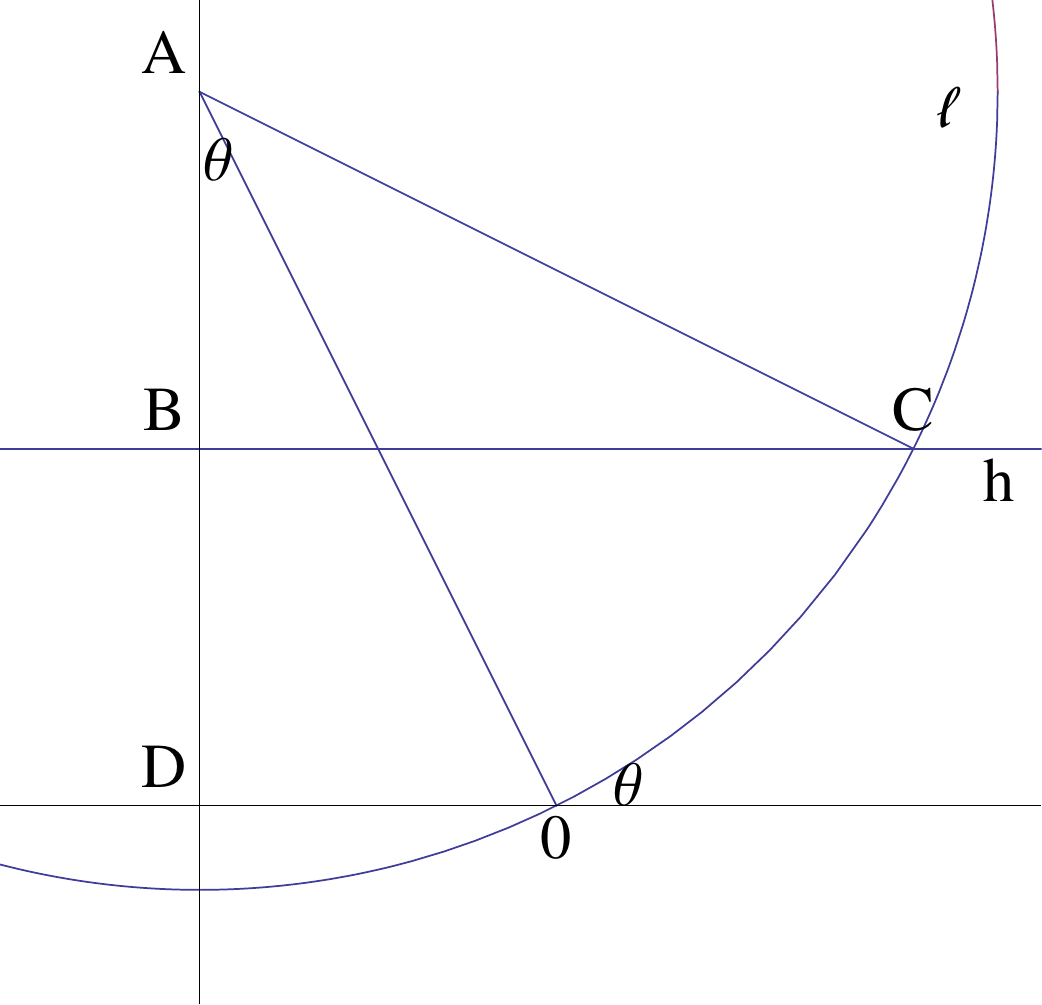}
\caption{A cleaner version of Fig. \ref{MAD}. The point $A$ is the center of the circle $\ell$, and $D$ denotes the origin.}
\label{xfig}
\end{figure}

Let $x=x(\gt)$ represent the length of $\overline{BC}$, and $y=y(\gt)$ be the distance from $B$ to $D$. We aim to compute the precise dependence of $x$ and $y$ on $\gt$. \\

We collect two identities for $\overline{AB}$:\\
\beann
\overline{AB} &=&\overline{AD}-\overline{BD}=R\cos\gt-y, \text{ and}\\
\overline{AB}^2 &=&\overline{AC}^2 - \overline{BC}^2=R^2-x^2.
\eeann

This implies
$$
x^2
=
R^2 - (R\cos\gt-y)^2 = R^2\sin^2\gt+2Ry\cos\gt-y^2
,
$$

\noindent
which together with \eqref{R} gives\\
\be\label{xth}
{x^2+y^2-1 \over 2y}={\cos\gt\over \sin\gt}\asymp {1\over \gt}
,
\ee
as $\gt\to0$.

Thus \eqref{xth}
, together with Proposition \ref{thbound},
%
%
 concludes the proof of Theorem \ref{eigHor}. 

 \bibliographystyle{alpha}

\bibliography{../../AKbibliog}

\end{document}